\documentclass[a4paper, reqno, 11pt]{amsart}
\usepackage[utf8]{inputenc}

\usepackage{amsthm}
\usepackage{amssymb}
\usepackage[american]{babel}
\usepackage{exscale}
\usepackage[mathscr]{euscript}
\usepackage{url}
\usepackage[all,ps,tips,tpic]{xy}
\usepackage{graphicx}
\usepackage{color}
\newenvironment{altenumerate}
   {\begin{list}
      {\textup{(\theenumi)} }
      {\usecounter{enumi}
       \setlength{\labelwidth}{0pt}
       \setlength{\labelsep}{2pt}
       \setlength{\leftmargin}{0pt}
       \setlength{\itemsep}{\the\smallskipamount}
       \renewcommand{\theenumi}{\roman{enumi}}
      }}
   {\end{list}}
\newenvironment{altitemize}
   {\begin{list}
      {$\bullet$ }
      {\setlength{\labelwidth}{0pt}
       \setlength{\labelsep}{2pt}
       \setlength{\leftmargin}{0pt}
       \setlength{\itemsep}{\the\smallskipamount}
      }}
   {\end{list}}

\newtheorem{lem}{Lemma}[section]

\newtheorem{definition}[lem]{Definition}

\newtheorem{cor}[lem]{Corollary}
\newtheorem{thm}[lem]{Theorem}
\newtheorem{prop}[lem]{Proposition}

\theoremstyle{remark}
\newtheorem{rem}[lem]{Remark}

\newcommand{\tr}{\operatorname*{tr}}
\newcommand{\GL}{\mathrm{GL}}

\newdir{ >}{{}*!/-5pt/@{>}}
\SelectTips{cm}{10}

\begin{document}
\title{The Local Langlands correspondence for $\GL_n$ over $p$-adic fields}
\author{Peter Scholze}
\begin{abstract}
We extend our methods from \cite{ScholzeGLn} to reprove the Local Langlands Correspondence for $\GL_n$ over $p$-adic fields as well as the existence of $\ell$-adic Galois representations attached to (most) regular algebraic conjugate self-dual cuspidal automorphic representations, for which we prove a local-global compatibility statement as in the book of Harris-Taylor, \cite{HarrisTaylor}.

In contrast to the proofs of the Local Langlands Correspondence given by Henniart, \cite{HenniartLLC}, and Harris-Taylor, \cite{HarrisTaylor}, our proof completely by-passes the numerical Local Langlands Correspondence of Henniart, \cite{HenniartNumericalLLC}. Instead, we make use of a previous result from \cite{ScholzeGLn} describing the inertia-invariant nearby cycles in certain regular situations.
\end{abstract}

\date{\today}
\maketitle
\tableofcontents
\pagebreak

\section{Introduction}

The aim of this paper is to give a new proof of the Local Langlands Correspondence for $\GL_n$ over $p$-adic fields, and to simplify some of the arguments in the book by Harris-Taylor, \cite{HarrisTaylor}.

Fix a $p$-adic field $F$, i.e., a finite extension of $\mathbb{Q}_p$, with ring of integers $\mathcal{O}$. Recall that the Local Langlands Correspondence, which is now a theorem due to Harris-Taylor, \cite{HarrisTaylor}, and Henniart, \cite{HenniartLLC}, asserts that there should be a canonical bijection between the set of isomorphism classes of irreducible supercuspidal representations of $\GL_n(F)$ and the set of isomorphism classes of irreducible $n$-dimensional representations of the Weil group $W_F$ of $F$, denoted $\pi\longmapsto \sigma(\pi)$. One possible local characterization of this bijection was given by Henniart, \cite{HenniartUniqueness}, showing that there is at most one family of bijections defined for all $n\geq 1$ preserving $L$- and $\epsilon$-factors of pairs, and also compatible with some basic operations on both sides such as twisting with characters.

Our starting point is a new local characterization of the Local Langlands Correspondence. First, we extend the map $\pi\mapsto \sigma(\pi)$ to all irreducible smooth representations $\pi$ by setting $\sigma(\pi)=\sigma(\pi_1)\oplus\ldots\oplus\sigma(\pi_t)$ if $\pi$ is a subquotient of the normalized parabolic induction of $\pi_1\otimes\cdots\otimes\pi_t$, for some irreducible supercuspidal representations $\pi_i$.\footnote{This is the Weil group representation underlying the Weil-Deligne representation attached to $\pi$. We ignore the monodromy operator in this article.} The rough idea is that one might hope to associate to any $\tau\in W_F$ a function $f_{\tau}\in C_c^{\infty}(\GL_n(F))$ such that for any irreducible smooth representation $\pi$, we have
\[
\tr(f_{\tau}|\pi) = \tr(\tau|\sigma(\pi))\ .
\]
Now, it is easy to see that this is quite a bit too much to hope for, as $f_{\tau}$ would have nonzero trace on all components of the Bernstein center. The best thing to hope for is to associate to some `cut-off' function $h\in C_c^{\infty}(\GL_n(F))$ a function $f_{\tau,h}\in C_c^{\infty}(\GL_n(F))$ such that for all irreducible smooth representations $\pi$, we have
\[
\tr(f_{\tau,h}|\pi) = \tr(\tau|\sigma(\pi)) \tr(h|\pi)\ .
\]
It turns out that there is a very natural way to do this, for many $\tau$ and $h$.

In the next few paragraphs, we assume that $F=\mathbb{Q}_p$ for simplicity. Recall that it is known that the cohomology of the Lubin-Tate tower realizes the Langlands Correspondence, cf. e.g. \cite{HarrisTaylor}, Theorem C, but only for supercuspidal representations. The idea is that replacing the Lubin-Tate space, i.e. the moduli space of one-dimensional formal groups of height $n$, by the moduli space of one-dimensional $p$-divisible groups of height $n$, one adds exactly the extra amount of information necessary to get the Langlands Correspondence for all irreducible smooth representations.

With this in mind, we just repeat the construction of the Lubin-Tate tower, except that we start with objects defined over a finite field: Take some integer $r\geq 1$ and a one-dimensional $p$-divisible group $\overline{H}$ of height $n$ over $\mathbb{F}_{p^r}$. Looking at its Dieudonn\'{e} module, this is equivalent to giving an element
\[
\beta\in \GL_n(\mathbb{Z}_{p^r})\mathrm{diag}(p,1,\ldots,1)\GL_n(\mathbb{Z}_{p^r})
\]
up to $\sigma$-conjugation by an element of $\GL_n(\mathbb{Z}_{p^r})$, where $\sigma$ is the absolute Frobenius of $\mathbb{Z}_{p^r}$.\footnote{It would be more customary to write $\delta$ instead of $\beta$, but following this convention would result in too many different $\delta$'s throughout this article.} Let $R_{\beta}$ be the deformation space of $\overline{H}$, with universal deformation $H$, and let $R_{\beta,m}/R_{\beta}$ be the covering parametrizing Drinfeld-level-$m$-structures on $H$. Then $\GL_n(\mathbb{Z}/p^m)$ acts on $R_{\beta,m}$. We choose $\ell\neq p$ and take the global sections of the nearby-cycle sheaves in the sense of Berkovich, \cite{Berkovich2}:\footnote{We caution the reader that we use the language used e.g. in SGA; in Berkovich's language, these sheaves would be called vanishing cycle sheaves.}
\[
R\psi_{\beta} = \mathop{\mathrm{lim}}\limits_{\longrightarrow} H^0(R\psi_{\mathrm{Spf}\ R_{\beta,m}}\bar{\mathbb{Q}}_{\ell})\ ,
\]
and their alternating sum $[R\psi_{\beta}]$. These objects carry an action of $W_{\mathbb{Q}_{p^r}}\times \GL_n(\mathbb{Z}_p)$. Now take an element $\tau\in W_{\mathbb{Q}_p}$ projecting to the $r$-th power of geometric Frobenius, and let $h\in C_c^{\infty}(\GL_n(\mathbb{Z}_p))$ have values in $\mathbb{Q}$. Define a new function $h^{\vee}\in C_c^{\infty}(\GL_n(\mathbb{Z}_p))$ by $h^{\vee}(g)=h((g^{-1})^t)$.

Our first theorem is the following.

\begin{thm}\label{MainTheorem1} Define a function $\phi_{\tau,h}$ on $\GL_n(\mathbb{Q}_{p^r})$ by
\[
\phi_{\tau,h}(\beta) = \tr(\tau\times h^{\vee}|[R\psi_{\beta}])\ ,
\]
if $\beta$ is as above, and by $0$ else. Then $\phi_{\tau,h}\in C_c^{\infty}(\GL_n(\mathbb{Q}_{p^r}))$, with values in $\mathbb{Q}$ independent of $\ell$.
\end{thm}

This allows us to define a function $f_{\tau,h}\in C_c^{\infty}(\GL_n(\mathbb{Q}_p))$ by requiring that it has matching (twisted) orbital integrals, cf. e.g. \cite{ArthurClozel}. We use the normalization of Haar measures that gives maximal compact subgroups volume $1$. Note that this function $f_{\tau,h}$ itself is not well-defined, but e.g. its orbital integrals and its traces on representations are.

For general $p$-adic fields $F$, we have an analogous definition of $f_{\tau,h}\in C_c^{\infty}(\GL_n(F))$ depending on $\tau\in W_F$ projecting to a positive power of geometric Frobenius and $h\in C_c^{\infty}(\GL_n(\mathcal{O}))$: One only has to replace $p$-divisible groups with $\varpi$-divisible $\mathcal{O}$-modules, i.e. $p$-divisible groups with $\mathcal{O}$-action over $\mathcal{O}$-schemes for which the two actions of $\mathcal{O}$ on the Lie algebra agree. Now we can state our second theorem.

\begin{thm}\label{MainTheorem2}
\begin{altenumerate}
\item[{\rm (a)}] For any irreducible smooth representation $\pi$ of $\GL_n(F)$ there is a unique $n$-dimensional representation $\mathrm{rec}(\pi)$ of $W_F$ such that for all $\tau$ and $h$ as above,
\[
\tr(f_{\tau,h}|\pi) = \tr(\tau|\mathrm{rec}(\pi)) \tr(h|\pi)\ .
\]
\end{altenumerate}

{\rm Write $\sigma(\pi) = \mathrm{rec}(\pi)(\frac{1-n}2)$.}

\begin{altenumerate}
\item[{\rm (b)}] If $\pi$ is a subquotient of the normalized parabolic induction of the irreducible representation $\pi_1\otimes\cdots \otimes \pi_t$ of $\GL_{n_1}(F)\times\cdots \times \GL_{n_t}(F)$, then $\sigma(\pi)=\sigma(\pi_1)\oplus\ldots \oplus \sigma(\pi_t)$.
\item[{\rm (c)}] The map $\pi\longmapsto \sigma(\pi)$ induces a bijection between the set of isomorphism classes of supercuspidal irreducible smooth representations of $\GL_n(F)$ and the set of isomorphism classes of irreducible $n$-dimensional representations of $W_F$.
\item[{\rm (d)}] The bijection defined in {\rm (c)} is compatible with twists, central characters, duals, and $L$- and $\epsilon$-factors of pairs, hence is the standard correspondence.
\end{altenumerate}
\end{thm}

The uniqueness assertion in {\rm (a)} is clear, as the condition determines $\tr(\tau|\mathrm{rec}(\pi))$ if $\tau$ projects to a positive power of geometric Frobenius, and it is well-known that this determines the representation. Hence this theorem gives a new local characterization of the Local Langlands Correspondence, which can be informally summarized by saying that the Local Langlands Correspondence is realized in the cohomology of the moduli space of one-dimensional $p$-divisible groups of height $n$, for all irreducible smooth representations. Also note that the extraneous division algebra acting in the Lubin-Tate setting disappears in our formulation.

The proof of this theorem occupies the whole paper, and consists of a local and a global part. Let us first say a few words about the global part of the story.

The global arguments are inspired by the work \cite{HarrisTaylor} of Harris-Taylor and there is a significant amount of overlap in the two approaches. In both cases, one considers the cohomology of the Shimura varieties associated to unitary groups of signature $(1,n-1)$ and split at $p$, uses it to construct $\ell$-adic Galois representations associated to certain cuspidal automorphic representations of $\GL_n$, and (most importantly) proves a local-global compatibility result at places of bad reduction. In the book of Harris-Taylor, this is achieved by introducing so-called Igusa varieties and counting points on them. This method has been extended to more general Shimura varieties through the work of Mantovan, \cite{Mantovan1}, \cite{Mantovan2}, and Shin, \cite{Shin}, and allows one to give expressions for the trace of arbitrary Hecke correspondences at $p$ on a Shimura variety.

In our approach, we restrict attention to Hecke operators at $p$ coming from the maximal compact subgroup. This has the technical advantage that one is pretty quickly reduced to calculating traces of genuine group actions, instead of just correspondences. In fact, one can work relatively over the Shimura variety with maximal compact level structure at $p$, and reduce all counting problems to counting problems for the maximal compact level structure, for which one can just appeal to the classical work of Kottwitz, \cite{KottwitzPoints}. With this approach, all of the difficult arguments concerning counting of points, stabilization, and using the trace formula can be borrowed from the work of Kottwitz.

At first sight, it may be surprising that it is enough to restrict attention to these Hecke operators. Indeed, one has to spice up the arguments with the use of type theory to get the arguments running: The necessary statements are recalled in Section \ref{TypeTheorySection}. All in all, we prove the following version of Theorem B of \cite{HarrisTaylor}. Fix an isomorphism $\bar{\mathbb{Q}}_{\ell}\cong \mathbb{C}$.

\begin{thm}\label{MainTheorem3} Let $\mathbb{F}$ be a CM field which is the composite of a totally real field $\mathbb{F}_0$ of even degree over $\mathbb{Q}$ and an imaginary-quadratic field. Fix a place $x$ of $\mathbb{F}$ which is split over $\mathbb{F}_0$. Let $\Pi$ be a cuspidal automorphic representation of $\GL_n/\mathbb{F}$ such that
\begin{altenumerate}
\item[{\rm (i)}] $\Pi^{\vee} = \Pi\circ c$, where $c: \GL_n(\mathbb{A}_\mathbb{F})\longrightarrow \GL_n(\mathbb{A}_\mathbb{F})$ is complex conjugation;
\item[{\rm (ii)}] $\Pi_{\infty}$ is regular algebraic, i.e. it has the same infinitesimal character as an algebraic representation of $\mathrm{Res}_{\mathbb{F}/\mathbb{Q}}(\GL_n)$ over $\mathbb{C}$;
\item[{\rm (iii)}] $\Pi_x$ is square-integrable.
\end{altenumerate}
Then there exists an integer $a\geq 1$ and an $\ell$-adic representation $R(\Pi)$ of $\mathrm{Gal}(\bar{\mathbb{F}}/\mathbb{F})$ of dimension $an$ such that for all finite places $v$ of $\mathbb{F}$ whose residue characteristic is different from $\ell$, we have
\[
R(\Pi)|_{W_{\mathbb{F}_v}} = a\cdot \mathrm{rec}(\Pi_v)
\]
as elements of the Grothendieck group of representations of $W_{\mathbb{F}_v}$.
\end{thm}

Using $p$-adic Hodge theory, it would be no problem to show that one can choose $a=1$, cf. proof of Proposition VII.1.8 in \cite{HarrisTaylor}. Also, the restrictions on $\mathbb{F}$ are unnecessary, cf. proof of Theorem VII.1.9 of \cite{HarrisTaylor}. Adding these extra arguments would reprove Theorem B of \cite{HarrisTaylor}.

With this theorem, one can almost prove part (a) of Theorem \ref{MainTheorem2}.

Now let us say a few words about the local arguments. It is easy to see that one can relate the deformation theory of a general one-dimensional $p$-divisible group to the deformation theory of its infinitesimal part, which gives some inductive formulas for the functions $f_{\tau,h}$. These allow us to almost prove part (b) of Theorem \ref{MainTheorem2}.

More importantly, one can reconstruct large parts of the cohomology of the Lubin-Tate tower from our function $f_{\tau,h}$, at least as far as the restriction of the $\GL_n(F)$-action to $\GL_n(\mathcal{O})$ is concerned. In fact, let $[R\psi]$ denote the alternating sum of the global sections of the nearby cycles for the Lubin-Tate tower. Then $[R\psi]$ carries an action of
\[
\GL_n(\mathcal{O})\times (\mathcal{D}^{\times}\times W_F)_0\ ,
\]
where $\mathcal{D}$ is the central division algebra over $F$ with invariant $\frac 1n$, and
\[
(\mathcal{D}^{\times}\times W_F)_0 = \{(d,\tau)\in \mathcal{D}^{\times}\times W_F\mid v(d) + v(\tau)=0\}\ ,
\]
and $v$ denotes natural valuations on $\mathcal{D}^{\times}$ and $W_F$. For any irreducible representation $\rho$ of $\mathcal{D}^{\times}$, the space
\[
[R\psi](\rho) = \mathrm{Hom}_{\mathcal{O}_\mathcal{D}^{\times}}(\rho|_{\mathcal{O}_\mathcal{D}^{\times}},[R\psi])
\]
carries an action of $\GL_n(\mathcal{O})\times W_F$. Let $\pi=\mathrm{JL}(\rho)$ be the associated representation of $\GL_n(F)$ via the Jacquet-Langlands correspondence.

\begin{thm}\label{MainTheorem4} Let $\rho$ be an irreducible representation of $\mathcal{D}^{\times}$ such that $\pi=\mathrm{JL}(\rho)$ is supercuspidal. Then, as a virtual representation of the group $\GL_n(\mathcal{O})\times W_F$, the representation $[R\psi](\rho^{\vee})$ is equal to $(-1)^{n-1} \pi^{\vee}|_{\GL_n(\mathcal{O})}\otimes \mathrm{rec}(\pi)$.
\end{thm}

After these steps, the main divergence to the known proofs occurs. In fact, it turns out that a previous result from \cite{ScholzeGLn} allows us to give a direct proof of part (c) of Theorem \ref{MainTheorem2}, i.e. the bijectivity of the correspondence. We do so without showing at this point that $\pi\longmapsto \sigma(\pi)$ preserves conductors or $L$- and $\epsilon$-factors, i.e. we do not make use of the numerical Local Langlands Correspondence of Henniart, \cite{HenniartNumericalLLC}. This argument is given in Section \ref{Bijectivity}. It relies on a geometric result from \cite{ScholzeGLn} describing the inertia-invariant nearby cycles in certain regular situations. This determines the inertia invariants $\sigma(\pi)^{I_F}$ for all irreducible smooth representations $\pi$, and implies that there are no supercuspidal representations that stay supercuspidal after any series of base-changes.

Hence, at this point we have proven parts (a) to (c) of Theorem \ref{MainTheorem2}, i.e. we have shown that with a natural local characterization of the Local Langlands Correspondence, this correspondence is unique, exists, and gives the desired bijection. Moreover, it is also compatible with global correspondences.

In fact, if one is only interested in parts (a) through (c) of Theorem \ref{MainTheorem2} for representations which are unitarily induced from supercuspidal, then global arguments are only used at the following points.\footnote{More precisely, one should say arguments involving Shimura varieties. Many statements from local harmonic analysis that are used in the local arguments, e.g. base-change of representations, are only proved by global means.} On the one hand, one needs Theorem \ref{MainTheorem4}, with some undetermined representation $\mathrm{rec}(\pi)$. The results of Strauch in \cite{Strauch} give some hope that this could be done locally. On the other hand, one needs the compatibility of $\pi\longmapsto \sigma(\pi)$ with base-change. Alternatively, one might try to prove compatibility with automorphic induction. For example, in the case where $F/\mathbb{Q}_p$ is a cyclic extension, the `forgetful' functor given by mapping a $\varpi$-divisible $\mathcal{O}$-module to its underlying $p$-divisible group gives a natural map from the moduli space of one-dimensional $\varpi$-divisible $\mathcal{O}$-modules of height $n$ to the moduli space of one-dimensional $p$-divisible groups of height $n[F:\mathbb{Q}_p]$, which seems to be related to automorphic induction.

We end the paper by using the method of non-Galois automorphic induction of Harris and the technique of twisting with highly ramified characters of Henniart to deduce that our correspondence satisfies the usual requirements on the Local Langlands Correspondence, i.e., we prove part (d) of Theorem \ref{MainTheorem2}.

Let us also mention that one can arrange the arguments so that our proof of the Local Langlands Correspondence does not makes use of the theory of types: One only proves parts (a) and (b) of Theorem \ref{MainTheorem2} for representations which are unitarily induced from supercuspidal, and only proves the local-global-compatibility result under this assumption. With this modification, the automorphic ingredients concerning Clozel's base-change from the unitary groups $\mathbf{G}$ to $\GL_n$ are only those that are also used implicitly (via reference to Harris' non-Galois automorphic induction, \cite{HarrisNonGalois}) by Henniart in his proof of the Local Langlands Correspondence, \cite{HenniartLLC}.

The Shimura varieties that we use in our arguments are very special, and indeed we use facts about them that are not true for general Shimura varieties. However, the construction of the function $\phi_{\tau,h}$, which in the present paper is based on Faltings' theory of group schemes with strict $\mathcal{O}$-action, \cite{Faltings}, taylored to our situation, works in much greater generality. The details, all going well, along with the generalization of Theorem \ref{TraceHeckeOperatorTO}, will appear in another paper, \cite{ScholzeDeformationSpaces}.

Finally, let us briefly summarize the content of the different sections. In Section 2, we define the deformation spaces of $p$-divisible groups and the functions $f_{\tau,h}$, and prove Theorem \ref{MainTheorem1}. We state Theorem \ref{MainTheorem2}, and the proof of parts (a) and (b) of this theorem will be given by induction on $n$ in Sections 3 -- 11. We begin by analyzing the necessary ingredients of the induction step in Section 3, thereby introducing certain statements (i), (ii) and (iii) in Lemma \ref{LemmaToProve}, which will be proved by separate methods.

Section 4 provides some basic statements about norm maps and base-change identities that will be useful throughout the text. Afterwards, in Sections 5 and 6, the proof of statement (i) is given, by relating the deformation spaces of general one-dimensional $p$-divisible groups to the deformation spaces of their infinitesimal parts. The geometric part of the argument is given in Section 5, and the harmonic analysis part in Section 6.

Next, we prove statement (ii) in Sections 7 -- 10 by a global argument. To set the stage for the global argument, we prove some preparatory local statements in Section 7. Then, in Section 8, we introduce the Shimura varieties that we will study, along with their integral models. The crucial counting argument is carried out in Section 9, and the results are applied in Section 10 to prove statement (ii) and Theorem \ref{MainTheorem3}.

The missing statement (iii), along with Theorem \ref{MainTheorem4}, is proved in Section 11, via the comparison with the Lubin-Tate tower.

After we have finished the proof of parts (a) and (b) of Theorem \ref{MainTheorem2}, we continue in Section 12 by using our earlier results from \cite{ScholzeGLn} to prove part (c). As already indicated above, in the final two Sections 13 and 14, we use Harris' method of non-Galois automorphic induction and Henniart's method of twisting with highly ramified characters, respectively, to prove part (d) of Theorem \ref{MainTheorem2}.

{\bf Notation.} We use $F$ to denote a finite extension of $\mathbb{Q}_p$ with ring of integers $\mathcal{O}$, uniformiser $\varpi$ and residue field $\kappa$. For any integer $r\geq 1$, we let $F_r/F$ be the unramified extension of degree $r$, with ring of integers $\mathcal{O}_r$ and residue field $\kappa_r$. The completion of the maximal unramified extension of $F$ is denoted $\breve{F}$, with ring of integers $\breve{\mathcal{O}}$. We use $\sigma$ to denote the arithmetic Frobenius of $F_r$ over $F$, or also of $\breve{F}$. In contrast, $\sigma_0$ denotes the arithmetic Frobenius of the Witt vectors $W(\kappa_r)$ over $\mathbb{Z}_p$.

Moreover, we denote the Weil group of $F$ by $W_F$, with inertia subgroup $I_F$, and we fix a geometric Frobenius element $\mathrm{Frob}\in W_F$.

To clarify the distinction, global objects are often denoted by bold-face letters, so that e.g. $\mathbb{F}$ will be a CM field with totally real subfield $\mathbb{F}_0$.

The symbol $^{\vee}$ is used to denote duals, e.g. dual representations, dual abelian varieties, and dual $p$-divisible groups. Also, if $f$ is a function on $\GL_n(R)$ for some ring $R$, we define a new function $f^{\vee}$ on $\GL_n(R)$ by $f^{\vee}(g)=f((g^{-1})^t)$.

{\bf Acknowledgments.} First of all, I thank my advisor M. Rapoport for explaining me the Langlands-Kottwitz method of counting points, which plays a crucial role in this article, for his encouragement to work on this topic, and for the many other things he taught me. Furthermore, my thanks go to Guy Henniart and Vincent S\'{e}cherre for their advice in type theory, among other things.

\section{Deformation spaces of $p$-divisible groups}\label{DefSpaces}

Recall the following definition, cf. e.g. \cite{Faltings}, where also the analogue of truncated $p$-divisible groups is defined.

\begin{definition} Let $S$ be an $\mathcal{O}$-scheme on which $p$ is locally nilpotent. A $\varpi$-divisible $\mathcal{O}$-module $H$ over $S$ is a $p$-divisible group $H$ over $S$ together with an action $\iota: \mathcal{O}\longrightarrow \mathrm{End}(H)$ such that the two induced actions of $\mathcal{O}$ on the Lie algebra of $H$ agree.
\end{definition}

Now let $H$ be a $\varpi$-divisible $\mathcal{O}$-module over a perfect field $k$ of characteristic $p$, which is given the structure of a $\mathcal{O}$-algebra, via a map $\kappa\longrightarrow k$. Then the usual Dieudonn\'{e} module $(M_0,F_0,V_0)$ of $H$ carries an action of
\[
\mathcal{O}\otimes W(k)=\prod_{\kappa\longrightarrow k} W_{\mathcal{O}}(k)\ ,
\]
where $W_{\mathcal{O}}(k)$ is the completion of the unramified extension of $\mathcal{O}$ with residue field $k$. Let $M$ be the component of $M_0$ corresponding to the given map $\kappa\longrightarrow k$, which is a free $W_{\mathcal{O}}(k)$-module. Assume that $\kappa\cong \mathbb{F}_{p^j}$ for some $j$. Then $M$ carries a $\sigma$-semilinear action of $F_0^j$, which we denote by $F$ in this context. One can check that $M$ also admits a $\sigma^{-1}$-semilinear operator $V$ satisfying $FV=VF=\varpi$. The structure $(M,F,V)$ is functorial in $H$ and is called the relative Dieudonn\'{e} module of $H$. It is an easy exercise to see that all of Dieudonn\'{e} theory goes through in this context.

In particular, to any $\beta\in \mathrm{GL}_n(\mathcal{O}_r)\mathrm{diag}(\varpi,1,\ldots,1)\mathrm{GL}_n(\mathcal{O}_r)$, one can associate a one-dimensional $\varpi$-divisible $\mathcal{O}$-module $\overline{H}_{\beta}$ of height $n$ over $\kappa_r$, by taking $F=\beta\sigma$. Conversely, any one-dimensional $\varpi$-divisible $\mathcal{O}$-module of height $n$ over $\kappa_r$ is associated to a unique $\GL_n(\mathcal{O}_r)$-$\sigma$-conjugacy class of such $\beta$.

\begin{definition}\label{DeformationSpaces} \begin{altenumerate}
\item[{\rm (i)}] Let $R_{\beta}$ be the formal deformation space of $\overline{H}_{\beta}$ as a $\varpi$-divisible $\mathcal{O}$-module, with universal deformation $H_{\beta}$.
\item[{\rm (ii)}] Let $R_{\beta,m}$ be the covering of $R_{\beta}$ parametrizing Drinfeld-level-$m$-structures on $H_{\beta}$, i.e. sections $X_1,\ldots,X_n\in H_{\beta}[\varpi^m]$ such that
\[
\sum_{i_1,\ldots,i_n\in \mathcal{O}/\varpi^m} [i_1X_1+\ldots+i_nX_n] = [H_{\beta}[\varpi^m]]
\]
as relative Cartier divisors on $H_{\beta}/R_{\beta}$.
\end{altenumerate}
\end{definition}

\begin{prop} \begin{altenumerate}
\item[{\rm (i)}] The ring $R_{\beta}$ is a formally smooth complete noetherian local $\mathcal{O}_r$-algebra, abstractly isomorphic to $\mathcal{O}_r[[T_1,...,T_{n-1}]]$.

\item[{\rm (ii)}] The covering $R_{\beta,m}/R_{\beta}$ is a finite Galois covering with Galois group $\GL_n(\mathcal{O}/\varpi^m\mathcal{O})$, \'{e}tale in the generic fibre.

\item[{\rm (iii)}] The ring $R_{\beta,m}$ is regular.
\end{altenumerate}
\end{prop}

\begin{proof} Using Proposition \ref{GeometricReduction}, one can reduce to the case where $H$ is infinitesimal. All statements can be checked after base-change to $\breve{\mathcal{O}}$, where these deformation spaces are the classical Lubin-Tate spaces and everything is well-known, cf. e.g. \cite{Drinfeld} or \cite{Strauch}.
\end{proof}

\begin{thm}\label{Algebraization} Associated to any double coset
\[
\overline{\beta}\in (1+\varpi^mM_n(\mathcal{O}))\backslash \GL_n(\mathcal{O}_r)\mathrm{diag}(\varpi,1,\ldots,1)\GL_n(\mathcal{O}_r)/(1+\varpi^mM_n(\mathcal{O}))\ ,
\]
there is a separated, flat scheme $\mathcal{R}_{\overline{\beta},m}$ of finite type over $\mathcal{O}_r$ with smooth generic fibre equipped with an action of $\GL_n(\mathcal{O}/\varpi^m)$, and a finite scheme $Z\subset \mathrm{Spec}\ \mathcal{R}_{\overline{\beta},m}\otimes_{\mathcal{O}_r} \kappa_r$ stable under this action such that the completion of $\mathcal{R}_{\overline{\beta},m}$ at $Z$ is $\GL_n(\mathcal{O})$-equivariantly isomorphic to $R_{\beta,m}$ for any $\beta\in \overline{\beta}$.
\end{thm}

\begin{proof} We recall the necessary facts from Faltings' theory of group schemes with strict $\mathcal{O}$-action, \cite{Faltings}. First, giving an $m$-truncated $\varpi$-divisible $\mathcal{O}$-module (or $\mathrm{BT}_m$) over $\kappa_r$ is equivalent to giving the $\sigma$-semilinear action of $F$ and the $\sigma^{-1}$-semilinear action of $V$ with $FV=VF=\varpi$ on a free $\mathcal{O}_r/\varpi^m$-module. Giving such data which come as the truncation of some $H_{\beta}$ is then seen to be equivalent to giving $\overline{\beta}$.

Further, Faltings shows that the functor $H\longmapsto H[\varpi^m]$ from $\varpi$-divisible $\mathcal{O}$-modules to $m$-truncated $\varpi$-divisible $\mathcal{O}$-modules is formally smooth, showing that $R_{\beta}$ is also the versal deformation space of $H_{\beta}[\varpi^m]$. Now this is a finitely presented functor, so that Artin's algebraization theorem, \cite{ArtinAlgebraization}, shows that there is a separated scheme $\mathcal{R}_{\overline{\beta}}$ of finite type over $\mathcal{O}_r$ together with an $m$-truncated $\varpi$-divisible $\mathcal{O}$-module $\mathcal{H}_{\overline{\beta}}$ and a point $x\in \mathrm{Spec}\ \mathcal{R}_{\overline{\beta}}(\kappa_r)$ such that the completion of $\mathcal{R}_{\overline{\beta}}$ at $x$ with $\mathcal{H}_{\overline{\beta}}$ restricted to this formal completion is isomorphic to $R_{\beta}$ with $H_{\beta}[\varpi^m]$ for all $\beta\in \overline{\beta}$.

By going over to a Zariski open subset and normalizing, we can assume that $\mathcal{R}_{\overline{\beta}}$ is normal, flat over $\mathcal{O}_r$ and smooth in the generic fibre, because these statements are true in the completion at $x$. Now let $\mathcal{R}_{\overline{\beta},m}$ be the normalization of $\mathcal{R}_{\overline{\beta}}$ in the covering of the generic fibre parametrizing trivializations of $\mathcal{H}_{\overline{\beta}}$. Then $\mathrm{GL}_n(\mathcal{O}/\varpi^m)$ acts on $\mathcal{R}_{\overline{\beta},m}$. Let $Z$ be the preimage of $x$ in $\mathcal{R}_{\overline{\beta},m}$.

The statements comparing $\mathcal{R}_{\overline{\beta},m}$ and $R_{\beta,m}$ are now clear, since $R_{\beta,m}$ is regular and in particular normal.

Finally, one may pass to a suitable Zariski open subset to assume that $\mathcal{R}_{\overline{\beta},m}$ is flat over $\mathcal{O}_r$ and smooth in the generic fibre, as these statements are true in the completion at $Z$.
\end{proof}

We consider the formal nearby cycle sheaves in the sense of Berkovich, \cite{Berkovich2},\footnote{Our terminology is the one used e.g. in SGA and differs from Berkovich's terminology, where these sheaves are called vanishing cycle sheaves.}
\[
R\psi_{\beta} = \mathop{\mathrm{lim}}\limits_{\stackrel{\longrightarrow}{m}} H^0(R\psi_{\mathrm{Spf}\ R_{\beta,m}}\bar{\mathbb{Q}}_{\ell})\ .
\]

\begin{thm}\label{RepresentationNice} The space $H^0(R^i\psi_{\mathrm{Spf}\ R_{\beta,m}} \bar{\mathbb{Q}}_{\ell})$ is a finite-dimensional, continuous representation of $W_{F_r}\times \GL_n(\mathcal{O}/\varpi^m)$, which vanishes outside the range $0\leq i\leq n-1$.
\end{thm}

\begin{proof} This follows from Theorem \ref{Algebraization}, the comparison of formal nearby cycles with the usual nearby cycles, i.e.,
\[
R^i\psi_{\mathrm{Spf}\ R_{\beta,m}} \bar{\mathbb{Q}}_{\ell} = R^i\psi_{\mathrm{Spec}\ \mathcal{R}_{\overline{\beta},m}} \bar{\mathbb{Q}}_{\ell}|_{Z\otimes_{\kappa_r} \bar{\kappa}}\ ,
\]
given by Theorem 3.1 of \cite{Berkovich2}, and the finiteness results in the \'{e}tale cohomology of schemes.
\end{proof}

We see that the alternating sum of the cohomology groups induces an element $[R\psi_{\beta}]$ in the Gro\-then\-dieck group of representations of $W_{F_r}\times \GL_n(\mathcal{O})$ with continuous $W_{F_r}$-action and smooth admissible $\GL_n(\mathcal{O})$-action.

Now, for any $\tau\in \mathrm{Frob}^r I_F\subset W_{F_r}$ and $h\in C_c^{\infty}(\GL_n(\mathcal{O}))$ taking values in $\mathbb{Q}$, we define
\[
\phi_{\tau,h}(\beta) = \tr(\tau\times h^{\vee}|[R\psi_{\beta}])\ ,
\]
where $h^{\vee}$ is defined by $h^{\vee}(g)=h((g^{-1})^t)$.

\begin{thm} This defines a function $\phi_{\tau,h}\in C_c^{\infty}(\GL_n(F_r))$ with values in $\mathbb{Q}$, independent of $\ell$.
\end{thm}

\begin{proof} From Theorem \ref{Algebraization}, it follows directly that the function is locally constant in $\beta$. The independence of $\ell$ follows from the results of Mieda, Theorem 6.2.2 of \cite{Mieda}. To apply them, assume that $h$ is the characteristic function of $g\in \GL_n(\mathcal{O}/\varpi^m)$. Use $g$ to twist the scheme $\mathcal{R}_{\overline{\beta},m}$ with the unramified action of $\mathrm{Gal}(\bar{F}/F)$ sending geometric Frobenius to $g$. We get a scheme $X$ with a finite subscheme $Z$ such that
\[
\phi_{\tau,h}(\beta) = \sum_{x\in Z(\kappa_r)} \tr(\tau | (R\psi_X \mathbb{Q}_{\ell})_x)\ .
\]
Now the claim is an immediate consequence of Theorem 6.2.2 of \cite{Mieda}.
\end{proof}

Finally, we let $f_{\tau,h}\in C_c^{\infty}(\GL_n(F))$ be associated to $\phi_{\tau,h}$, using the normalization of Haar measures giving a maximal compact subgroup volume $1$. Recall our second main theorem.

\begin{thm}
\begin{altenumerate}
\item[{\rm (a)}] For any irreducible smooth representation $\pi$ of $\GL_n(F)$ there is a unique $n$-dimensional representation $\mathrm{rec}(\pi)$ of $W_F$ such that for all $\tau$ and $h$ as above,
\[
\tr(f_{\tau,h}|\pi) = \tr(\tau|\mathrm{rec}(\pi)) \tr(h|\pi)\ .
\]
\end{altenumerate}

{\rm Write $\sigma(\pi) = \mathrm{rec}(\pi)(\frac{1-n}2)$.}

\begin{altenumerate}
\item[{\rm (b)}] If $\pi$ is a subquotient of the normalized parabolic induction of the irreducible representation $\pi_1\otimes\cdots \otimes \pi_t$ of $\GL_{n_1}(F)\times\cdots \times \GL_{n_t}(F)$, then $\sigma(\pi)=\sigma(\pi_1)\oplus\ldots \oplus \sigma(\pi_t)$.
\item[{\rm (c)}] The map $\pi\longmapsto \sigma(\pi)$ induces a bijection between the set of isomorphism classes of supercuspidal irreducible smooth representations of $\GL_n(F)$ and the set of isomorphism classes of irreducible $n$-dimensional representations of $W_F$.
\item[{\rm (d)}] The bijection defined in {\rm (c)} is compatible with twists, central characters, duals, and $L$- and $\epsilon$-factors of pairs, hence is the standard correspondence.
\end{altenumerate}
\end{thm}

\section{An application of the theory of types}\label{TypeTheorySection}

If $\pi$ is an admissible smooth representation of $\GL_n(F)$ and $s\in \mathbb{C}$, then we let $\pi[s] = \pi\otimes |\mathrm{det}|^s$ be the twisted representation.

For any collection $\pi_1,\ldots,\pi_t$ of irreducible essentially square-integrable representations of $\GL_{n_i}(F)$ such that $\pi_i$ does not precede $\pi_j$ for $i<j$ in the sense of Bernstein-Zelevinsky, \cite{BernsteinZelevinsky}, \cite{Zelevinsky}, we denote by $\boxplus_{i=1}^t \pi_t$ the Langlands quotient of the normalized parabolic induction of $\pi_1\otimes\cdots\otimes\pi_t$. It is known that this does not depend on the ordering of the $\pi_i$. We extend the definition to any collection of $\pi_i$ by first reordering them so that $\pi_i$ does not precede $\pi_j$ if $i<j$, which is always possible.

Let $d$ be a divisor of $n$, hence $n=dt$ for some integer $t\geq 1$, and let $\pi_0$ be a unitary irreducible supercuspidal representation of $\GL_k(F)$. Then we call $\pi=\boxplus_{i=1}^t \pi_0[\frac{t+1}2 - i]$ a generalized Speh representation of $\GL_n(F)$.

\begin{lem}\label{TypeTheory} {\rm (i)} Let $n=dt$ and $\pi_0$ be as above and let $\pi=\boxplus_{i=1}^t \pi_0[\frac{t+1}2 - i]$. Assume that $t\geq 2$. Then there exists a function $h\in C_c^{\infty}(\GL_n(\mathcal{O}))$ such that $\tr(h|\rho)=0$ for all irreducible tempered representations $\rho$ of $\GL_n(F)$ which are not of the form $\rho=\boxplus_{i=1}^t \pi[s_i]$ for some numbers $s_i\in \mathbb{C}$ of real part $0$, and for which $\tr(h|\pi)\neq 0$.

{\rm (ii)} Assume that $h\in C_c^{\infty}(\GL_n(F))$ is such that for all irreducible tempered, non-square-integrable representations $\pi$ of $\GL_n(F)$ and for all generalized Speh representations $\pi$ we have $\tr(h|\pi)=0$. Then $\tr(h|\pi)=0$ for all irreducible smooth representations $\pi$ of $\GL_n(F)$.
\end{lem}

\begin{proof} Part (i) follows from the results of Schneider and Zink in \cite{SchneiderZinkTypes}. We deduce it from Proposition 11 of \cite{SchneiderZinkTypes}. This shows that there exists an irreducible representation $\tau$ (equal to $\sigma_{\mathcal{P}}(\lambda)$ for $\mathcal{P}$ minimal in their notation) of $\GL_n(\mathcal{O})$ that occurs in a tempered representation $\rho$ if and only if $\rho=\boxplus_{i=1}^t \pi[s_i]$ for some numbers $s_i\in \mathbb{C}$ of real part $0$. We let $h$ be the character of $\tau^{\vee}$. Then the first part of (i) holds true. But Proposition 11, part (i), of \cite{SchneiderZinkTypes} says that $\tau$ also occurs in $\pi$, which implies the second part of (i).

For part (ii), we first show that $\tr(h|\pi)=0$ for all properly induced representations $\pi$. Indeed, let $P$ be a maximal parabolic with Levi $M=\GL_k\times \GL_{n-k}$, and let $h^P$ be the (normalized, $K$-invariant) constant term along $P$, so that
\[
\tr(h|\text{n-Ind}_{P(F)}^{\GL_n(F)} \pi_1\otimes\pi_2)=\tr(h^P|\pi_1\otimes\pi_2)
\]
for all admissible smooth representations $\pi_1$, $\pi_2$ of $\GL_k(F)$ resp. $\GL_{n-k}(F)$. Then our assumptions say that $\tr(h^P|\pi_1\otimes\pi_2)=0$ for all irreducible tempered $\pi_1$, $\pi_2$, hence by Kazhdan's density theorem, \cite{Kazhdan}, this also holds for all $\pi_1$, $\pi_2$, which is what we have claimed.

Now, if $\pi$ is square-integrable, then $\pi$ corresponds to some segment
\[
I=[\pi_0(\tfrac{1-t}2),\pi_0(\tfrac{t-1}2)]\ .
\]
But one sees e.g. from Lemma I.3.2 of \cite{HarrisTaylor} that $\pi + (-1)^t \boxplus_{i=1}^t \pi_0[\frac{t+1}2 -i]$ is an alternating sum of induced representations in the Grothendieck group of admissible representations of $\GL_n(F)$. This shows that $\tr(h|\pi)=0$, hence $\tr(h|\pi)=0$ for all tempered $\pi$ and thus by Kazhdan's density theorem for all $\pi$.
\end{proof}

\begin{lem}\label{LemmaToProve} Assume that the following conditions are satisfied:
\begin{altenumerate}
\item[{\rm (o)}] Parts (a) and (b) of Theorem \ref{MainTheorem2} hold true for all $n^{\prime}<n$.
\item[{\rm (i)}] If $\pi$ is the (not necessarily irreducible) normalized parabolic induction of $\pi_1\otimes\cdots\otimes\pi_t$, where $t\geq 2$ and $\pi_i$ is an irreducible smooth representation of $\GL_{n_i}(F)$, then for all $\tau$, $h$, we have
\[
\tr(f_{\tau,h} | \pi) = \tr\left(\tau|\mathrm{rec}(\pi_1)(\tfrac{n-n_1}2)\oplus\ldots\oplus \mathrm{rec}(\pi_t)(\tfrac{n-n_t}2)\right) \tr(h|\pi)\ .
\]
\item[{\rm (ii)}] For any irreducible smooth representation $\pi$ of $\GL_n(F)$ which is either essentially square-integrable or a generalized Speh representation, there exists a $\mathbb{Q}$-linear combination $\mathrm{rec}(\pi)$ of representations of $W_F$ with positive coefficients and total dimension $n$ such that for all $\tau$, $h$, we have
\[
\tr(f_{\tau,h}|\pi) = \tr(\tau|\mathrm{rec}(\pi)) \tr(h|\pi)\ .
\]
\item[{\rm (iii)}] If $\pi$ is a supercuspidal irreducible smooth representation of $\GL_n(F)$, then $\mathrm{rec}(\pi)$ is a $\mathbb{Z}$-linear combination of representations of $W_F$.
\end{altenumerate}
Then parts (a) and (b) of Theorem \ref{MainTheorem2} hold true for $n$.
\end{lem}

\begin{proof} First note that parts (ii) and (iii) imply that for $\pi$ supercuspidal $\mathrm{rec}(\pi)$ is an $n$-dimensional representation of $W_F$: Writing $\mathrm{rec}(\pi)$ in the basis given by the irreducible representations of $W_F$, we know that all coefficients have to be nonnegative by (ii) and integers because of (iii).

If $\pi$ is an irreducible smooth representation of $\GL_n(F)$ with supercuspidal support $\pi_1,\ldots,\pi_t$, write $\sigma(\pi) = \mathrm{rec}(\pi_1)(\frac{1-n_1}2)\oplus\ldots\oplus \mathrm{rec}(\pi_t)(\frac{1-n_t}2)$, which is an $n$-dimensional representation of $W_F$. Let $f_{\tau}$ be the function of the Bernstein center that acts through the scalar $\tr(\tau|\sigma(\pi)(\frac{n-1}2))$ on any irreducible smooth representation $\pi$. To prove the existence of $f_{\tau}$, we have to show that
\[
\tr(\tau|\sigma(\pi)(\tfrac{n-1}2))
\]
defines a regular function on any component of the Bernstein center. It is clear that the function depends only on the supercuspidal support, so that one immediately reduces to the case of a supercuspidal component, given by unramified twists of some supercuspidal representation $\pi$. In that case, choose $h\in C_c^{\infty}(\GL_n(\mathcal{O}))$ with $\tr(h|\pi)=1$. Then $\tr(h|\pi^{\prime})=1$ for all unramified twists $\pi^{\prime}$ of $\pi$, and hence
\[
\tr(\tau|\sigma(\pi^{\prime})(\tfrac{n-1}2)) = \tr(\tau|\mathrm{rec}(\pi^{\prime})) = \tr(f_{\tau,h}|\pi^{\prime})\ .
\]
But the right-hand side does indeed give a regular function on this Bernstein component, as does the trace of any $C_c^{\infty}$-function.

It is enough to see that
\[
\tr(f_{\tau,h}|\pi) = \tr(f_{\tau}\ast h|\pi)
\]
for all irreducible smooth representations $\pi$ of $\GL_n(F)$.

If $\pi$ is properly induced or supercuspidal, this is clear because of condition (i), resp. (ii). By part (ii) of Lemma \ref{TypeTheory}, it is enough to check it for all generalized Speh representations $\pi = \boxplus_{i=1}^r \pi_0[\frac{r+1}2 - i]$, where $\pi_0$ is a unitary irreducible supercuspidal representation of a smaller $\GL_d(F)$. Choose $h$ as in part (i) of Lemma \ref{TypeTheory}. Then we see that
\[
\tr(f_{\tau,h}|\pi^{\prime}) = \tr(f_{\tau}\ast h|\pi^{\prime})
\]
for all tempered representations $\pi^{\prime}$: For all except the square-integrable ones, we have already seen this, and for square-integrable representations, it holds true because both sides are scalar multiples of $\tr(h|\pi^{\prime})$, which vanishes by choice of $h$. Hence this equality also holds true for any other irreducible smooth representation, in particular for $\pi=\boxplus_{i=1}^r \pi_0[\frac{r+1}2 - i]$. In that case, the equality says that
\[
\mathrm{tr}(\tau|\mathrm{rec}(\pi)) = \mathrm{tr}\left(\tau|\sigma(\pi)(\tfrac{n-1}2)\right)\ .
\]
This implies that $\mathrm{rec}(\pi) = \sigma(\pi)(\frac{n-1}2)$, and hence we know that for all $h$, the equality
\[
\tr(f_{\tau,h}|\pi) = \tr(f_{\tau}\ast h|\pi)
\]
holds for all generalized Speh representations. As indicated above, this finishes the proof by part (ii) of Lemma \ref{TypeTheory}.
\end{proof}

\section{Norm maps and $p$-divisible groups}\label{NormMaps}

In this section, we will study norm maps on integral elements by relating these to different parametrizations of $p$-divisible groups. This will give a very concise proof of base-change lemmas.

First, we consider the case of \'{e}tale $p$-divisible groups. Let $D$ be a semisimple algebra over $\mathbb{Q}_p$ with a maximal order $\mathcal{O}_D$.

\begin{definition}\label{DefBGroup} Let $S$ be scheme on which $p$ is locally nilpotent. A $D$-group over $S$ is an \'{e}tale $p$-divisible group $H$ over $S$ together with an action
\[
\iota: \mathcal{O}_D^{\mathrm{op}}\longrightarrow \mathrm{End}(H)
\]
such that $H[p]$ is free of rank $1$ over $\mathcal{O}_D^{\mathrm{op}}/p$.
\end{definition}

There are two ways to parametrize $D$-groups $H$ over $\mathbb{F}_{p^r}$. On the one hand, one can look at the (contravariant) Dieudonn\'{e} module $M$ of $H$; then $M\cong \mathcal{O}_D\otimes_{\mathbb{Z}_p} \mathbb{Z}_{p^r}$. If we write $\sigma_0$ for the absolute Frobenius of $\mathbb{Z}_{p^r}$, then one can write the Frobenius $F$ of $M$ as $F=\beta^{-1}\sigma_0$ for some $\beta\in (\mathcal{O}_D\otimes_{\mathbb{Z}_p} \mathbb{Z}_{p^r})^{\times}$. This element is well-defined up to $\sigma_0$-conjugation by an element of $(\mathcal{O}_D\otimes_{\mathbb{Z}_p} \mathbb{Z}_{p^r})^{\times}$. We call this the Dieudonn\'{e} parametrization.

On the other hand, one can first give a $D$-group $\tilde{H}$ over $\bar{\mathbb{F}}_p$ and then add a descent datum to $\mathbb{F}_{p^r}$. Let $\mathrm{Frob}\in \mathrm{Gal}(\bar{\mathbb{F}}_p/\mathbb{F}_{p^r})$ be the geometric Frobenius; then a descent datum is given by an isomorphism $\alpha: \mathrm{Frob}^{\ast}\tilde{H}\cong \tilde{H}$, giving the action of $\mathrm{Frob}$. Let
\[
F: \bar{\mathbb{F}}_p\longrightarrow \bar{\mathbb{F}}_p
\]
be the $p$-th power map; then there is the natural Frobenius isogeny $F: \tilde{H}\longrightarrow F^{\ast}\tilde{H}$ defined for any $p$-divisible group, which is an isomorphism in the case of \'{e}tale $p$-divisible groups. Then giving a descent datum $\alpha$ is equivalent to giving $\gamma=\alpha\circ F^{-r}: \tilde{H}\longrightarrow \tilde{H}$, which is an $\mathcal{O}_D^{\mathrm{op}}$-linear automorphism. Since $\mathrm{End}_{\mathcal{O}_D^{\mathrm{op}}}(\tilde{H})=\mathcal{O}_D$, it follows that giving a descent datum is equivalent to giving an element $\gamma\in \mathcal{O}_D^{\times}$. One easily checks that $\gamma$ is well-defined up to $\mathcal{O}_D^{\times}$-conjugation. We call this the Galois parametrization.

\begin{prop} This defines a bijection between the set of $(\mathcal{O}_D\otimes_{\mathbb{Z}_p} \mathbb{Z}_{p^r})^{\times}$-$\sigma_0$-conjugacy classes in $(\mathcal{O}_D\otimes_{\mathbb{Z}_p} \mathbb{Z}_{p^r})^{\times}$ and the set of $\mathcal{O}_D^{\times}$-conjugacy classes in $\mathcal{O}_D^{\times}$, which we denote by $\beta\longmapsto N\beta$.

Further, the characteristic polynomials of $N\beta$ and $\beta\beta^{\sigma_0}\cdots\beta^{\sigma_0^{r-1}}$ agree. 
\end{prop}

\begin{proof} The first part is clear by construction. For the second part, look at the Dieudonn\'{e} module of $\tilde{H}$: It is given by
\[
M\otimes_{\mathbb{Z}_{p^r}} W(\bar{\mathbb{F}}_p)\ ,
\]
and the action of $\mathrm{Frob}$ is given by its action on the second factor, where it acts through $\mathrm{Frob}^{-1}$ by contravariance. The Frobenius isogeny is given by $F$, and hence $\alpha\circ F^{-r}$ is given by
\[
\gamma=F^{-r}\otimes\mathrm{Frob}^{-1}=(\beta^{-1}\sigma_0)^{-r}\otimes \sigma_0^r=N\beta\ .
\]
For the last equality, note that the evaluation of $(\beta^{-1}\sigma_0)^{-r}$ takes place in the algebra $\mathcal{O}_D^{\mathrm{op}}\otimes_{\mathbb{Z}_p} W(\bar{\mathbb{F}}_p)$ of $\mathcal{O}_D$-linear endomorphisms of $M\otimes_{\mathbb{Z}_{p^r}} W(\bar{\mathbb{F}}_p)$. This implies the desired statement.
\end{proof}

Now assume that $h\in C_c^{\infty}(\mathcal{O}_D^{\times})$ is invariant under conjugation. We get a function $\phi_h\in C_c^{\infty}((\mathcal{O}_D\otimes_{\mathbb{Z}_p} \mathbb{Z}_{p^r})^{\times})$ by setting $\phi_h(\beta)=h(N\beta)$.

We choose Haar measures on $D^{\times}$ and $(D\otimes_{\mathbb{Q}_p} \mathbb{Q}_{p^r})^{\times}$ that give maximal compact subgroups volume $1$.

\begin{prop}\label{BGroupAssociatedFcts} In this situation, the functions $\phi_h$ and $h$ have matching (twisted) orbital integrals.
\end{prop}

\begin{rem} Taking $D$ to be $M_n(\mathbb{Q}_p)$ and $h$ to be characteristic function of $\GL_n(\mathbb{Z}_p)$, this gives the usual base-change identity.
\end{rem}

\begin{proof} We have to show that for any $\beta\in (\mathcal{O}_D\otimes_{\mathbb{Z}_p}\mathbb{Z}_{p^r})^{\times}$,
\[
TO_{\beta\sigma_0}(\phi_h) = O_{N\beta}(h)\ .
\]
Let $H$ be the $D$-group over $\mathbb{F}_{p^r}$ associated to $\beta$. Consider the set $X$ of $D$-groups $H^{\prime}$ over $\mathbb{F}_{p^r}$ together with an $\mathcal{O}_D^{\mathrm{op}}$-linear quasi-isogeny $\alpha: H^{\prime}\longrightarrow H$. On this set $X$, we have an action of $\Gamma = (\mathrm{End}(H)\otimes_{\mathbb{Z}_p} \mathbb{Q}_p)^{\times}$ by composition.

First, it is easy to see that for any $x\in X$, the stabilizer $\Gamma_x\subset \Gamma$ is a maximal compact subgroup. This shows that all $\Gamma_x$ have the same volume. We may normalize the Haar measure by requiring that these subgroups have volume $1$.

Note that we can define a $\Gamma$-invariant function $\tilde{h}$ on $X$ by requiring $\tilde{h}(H^{\prime},\alpha) = h(\gamma(H^{\prime}))$. We claim that
\[
TO_{\beta\sigma_0}(\phi_h) = O_{N\beta}(h) = \sum_{x\in X/\Gamma} \tilde{h}(x)\ .
\]

Now on the one hand, Dieudonn\'{e} theory gives an isomorphism
\[
\Gamma = \{g\in (D\otimes_{\mathbb{Q}_p} \mathbb{Q}_{p^r})^{\times}\mid g^{-1}\beta g^{\sigma}=\beta\}
\]
and an identification of $X$ with
\[
\{g\in (D\otimes_{\mathbb{Q}_p} \mathbb{Q}_{p^r})^{\times}\mid g^{-1}\beta g^{\sigma}\in \mathcal{O}_D^{\times} \}/(\mathcal{O}_D\otimes_{\mathbb{Z}_p} \mathbb{Z}_{p^r})^{\times}\ ,
\]
by sending $g$ to the $D$-group $H^{\prime}$ given by the lattice $gM\subset M\otimes_{\mathbb{Z}_p}\mathbb{Q}_p$ and the corresponding quasi-isogeny $\alpha:H^{\prime}\longrightarrow H$. Also $\tilde{h}(H^{\prime},\alpha) = \phi_h(g^{-1}\beta g^{\sigma})$ under this correspondence. This proves the equality
\[
TO_{\beta\sigma_0}(\phi_h) = \sum_{x\in X/\Gamma} \tilde{h}(x)\ .
\]

Reasoning in the same way with the other parametrization of $X$, we get the second identity.
\end{proof}

\begin{cor}\label{IntegrationFormulaEtale} With $h$ and $\phi_h$ as above, we have
\[
\int_{(\mathcal{O}_D\otimes_{\mathbb{Z}_p} \mathbb{Z}_{p^r})^{\times}} \phi_h(\beta) d\beta = \int_{\mathcal{O}_D^{\times}} h(\gamma) d\gamma\ .
\]
\end{cor}

\begin{proof} This follows from the Weyl integration formula, since $h$ and $\phi_h$ are associated.
\end{proof}

Now we repeat these arguments for the case of one-dimensional formal $\mathcal{O}$-modules of height $n$.

Write $\mathcal{D}$ for the central divison algebra over $F$ with invariant $\frac 1n$, with ring of integers $\mathcal{O}_\mathcal{D}$. There is a natural valuation $v: \mathcal{D}^{\times}\longrightarrow \mathbb{Z}$, taking $\varpi\in F\subset \mathcal{D}$ to $n$. Let $B_r$ be the set of basic elements in $\GL_n(\mathcal{O}_r)\mathrm{diag}(\varpi,1,\ldots,1)\GL_n(\mathcal{O}_r)$, by which we mean those elements that are basic as elements of $\GL_n(\breve{F})$. Also, we let $\mathcal{D}_r$ be the set of elements of $\mathcal{D}^{\times}$ whose valuation is $r$.

\begin{prop} There is a natural bijection between $\GL_n(\mathcal{O}_r)$-$\sigma$-conjugacy classes in $B_r$ and $\mathcal{O}_\mathcal{D}^{\times}$-conjugacy classes in $\mathcal{D}_r$, denoted $\beta\longmapsto N\beta$. Further, the characteristic polynomials of $N\beta$ and $\beta\beta^{\sigma}\cdots\beta^{\sigma^{r-1}}$ agree.
\end{prop}

\begin{proof} Consider the set of formal one-dimensional $\varpi$-divisible $\mathcal{O}$-modules over $\kappa_r$ of height $n$; these are in bijection to each of the two sets. For the first, this is clear by looking at the relative Dieudonn\'{e} module, writing $F=\beta\sigma$. For the second, note that over $\bar{\kappa}$, there is a unique formal one-dimensional $\varpi$-divisible $\mathcal{O}$-module $\tilde{H}$ of height $n$, for which $\mathrm{End}(\tilde{H})=\mathcal{O}_\mathcal{D}$. Under this identification, the valuation on $\mathcal{O}_\mathcal{D}$ is given by the height of endomorphisms of $\tilde{H}$. Giving the descent datum to $\kappa_r$ is equivalent to giving an element of $\mathcal{D}_{-r}$, using the same reasoning as in the \'{e}tale case and the fact that the Frobenius has height $1$. We identify elements of $\mathcal{D}_{-r}$ with elements of $\mathcal{D}_r$ via the $d\mapsto d^{-1}$. The final statement is seen by looking at the Dieudonn\'{e} module of $\tilde{H}$ as in the case of \'{e}tale $p$-divisible groups.
\end{proof}

For any function $h\in C_c^{\infty}(\mathcal{D}_r)$ which is invariant under $\mathcal{O}_\mathcal{D}^{\times}$-conjugation, we define $\phi_h\in C_c^{\infty}(B_r)$ by $\phi_h(\beta)=h(N\beta)$. The next two statements are proved in the same way as the corresponding statements in the \'{e}tale case.

\begin{prop} The functions $\phi_h$ and $h$ have matching (twisted) orbital integrals.
\end{prop}

\begin{cor}\label{IntegrationLemma} We have the equality
\[
\int_{B_r} \phi_h(\beta) d\beta = \int_{\mathcal{D}_r} h(d) dd\ .
\]
\end{cor}

\section{Descent properties of the test function: Geometry}

Pick some
\[
\beta\in \GL_n(\mathcal{O}_r)\mathrm{diag}(\varpi,1,\ldots,1)\GL_n(\mathcal{O}_r)
\]
and consider the associated $\varpi$-divisible $\mathcal{O}$-module $\overline{H}_{\beta}$ over $\kappa_r$. Write
\[
\overline{H}_{\beta}=\overline{H}_{\beta}^0 \times \overline{H}_{\beta}^{\mathrm{et}}
\]
as the product of its infinitesimal and \'{e}tale part. Let $k$ be the height of $\overline{H}_{\beta}^0$ as a formal $\mathcal{O}$-module. After $\sigma$-conjugation, we correspondingly get
\[
\beta=(\beta^0,\beta^{\mathrm{et}})\in \big(GL_k(\mathcal{O}_r)\mathrm{diag}(\varpi,1,\ldots,1)\GL_k(\mathcal{O}_r)\big)\times \GL_{n-k}(\mathcal{O}_r)\ .
\]
Recall that one can also use the Galois parametrization for $\overline{H}_{\beta}^{\mathrm{et}}$. Repeating the arguments on \'{e}tale $p$-divisible groups from Section \ref{NormMaps} with relative Dieudonn\'{e} modules replacing usual Dieudonn\'{e} modules (and keeping track of normalizations) shows that the action of $\mathrm{Frob}$ on $\overline{H}_{\beta}^{\mathrm{et}}(\bar{k})$ is through right multiplication by $\gamma=(N\beta^{\mathrm{et}})^{-1}$.

Let $\breve{\mathcal{O}}$ be the completion of the maximal unramified extension of $\mathcal{O}$.

\begin{prop}\label{GeometricReduction} \begin{altenumerate}
\item[{\rm (i)}] The connected components of $\mathrm{Spf}\ (R_{\beta,m}\otimes_{\mathcal{O}_r} \breve{\mathcal{O}})$ are parametrized by surjective $\mathcal{O}$-linear maps
\[
v: (\mathcal{O}/\varpi^m)^n\longrightarrow \overline{H}_{\beta}^{\mathrm{et}}[\varpi^m](\bar{k})\ .
\]
In particular, $\GL_n(\mathcal{O}/\varpi^m)$ acts transivitely on the set of connected components.
\item[{\rm (ii)}] Let $V\subset (\mathcal{O}/\varpi^m)^n$ be a direct summand of rank $k$ over $\mathcal{O}/\varpi^m$. Let
\[
\mathrm{Spf}\ (R_{\beta,m}\otimes_{\mathcal{O}_r}\breve{\mathcal{O}})_V
\]
be the union of all connected components corresponding to $v$ with $\ker v=V$. Then $(R_{\beta,m}\otimes_{\mathcal{O}_r}\breve{\mathcal{O}})_V$ is defined over $\mathcal{O}_r$; call it $(R_{\beta,m})_V$. Further, the stabilizer of $(R_{\beta,m})_V\subset R_{\beta,m}$ is $P_V(\mathcal{O}/\varpi^m)$, where $P_V$ is the parabolic associated to $V$, i.e. the subgroup stabilizing $V$.
\end{altenumerate}

{\rm In the following, we take $V=(\mathcal{O}/\varpi^m)^k\subset (\mathcal{O}/\varpi^m)^n$. We write $P_k=P_V$.}

\begin{altenumerate}
\item[{\rm (iii)}] The canonical projection
\[
\mathrm{Spf}\ (R_{\beta,m})_V\longrightarrow \mathrm{Spf}\ R_{\beta^0,m}\ ,
\]
given by sending a deformation of $\overline{H}_{\beta}$ with Drinfeld-level structure $(X_1,\ldots,X_n)$ to its infinitesimal part and the Drinfeld-level-structure given by $(X_1,\ldots,X_k)$, is formally smooth, inducing (noncanonical) isomorphisms of each connected component of
\[
\mathrm{Spf}\  (R_{\beta,m})_V\otimes_{\mathcal{O}_r} \breve{\mathcal{O}}
\]
with
\[
\mathrm{Spf}\  (R_{\beta^0,m}\otimes_{\mathcal{O}_r}\breve{\mathcal{O}})[[T_1,...,T_{n-k}]]\ .
\]
\item[{\rm (iv)}] The map in {\rm (iii)} induces a (canonical) $P_k(\mathcal{O}/\varpi^m)\times W_{F_r}$-equivariant isomorphism
\[
R\psi_{\mathrm{Spf}\ (R_{\beta,m})_V} \bar{\mathbb{Q}}_{\ell}\cong R\psi_{\mathrm{Spf}\ R_{\beta^0,m}} \otimes \bar{\mathbb{Q}}_{\ell}[\GL_{n-k}(\mathcal{O}/\varpi^m)]\ ,
\]
where the $P_k(\mathcal{O}/\varpi^m)$-action factors through its Levi quotient, and the $W_{F_r}$-action is the natural action on both nearby cycle sheaves and is the unramified action taking $\mathrm{Frob}^r$ to right multiplication by $(\gamma^{-1})^t = (N\beta^{\mathrm{et}})^t$ on $\bar{\mathbb{Q}}_{\ell}[\GL_{n-k}(\mathcal{O}/\varpi^m)]$.
\item[{\rm (v)}] This gives an equality of virtual $\GL_n(\mathcal{O})\times W_{F_r}$-representations
\[
[R\psi_{\beta}]=\mathrm{Ind}_{P_k(\mathcal{O})}^{\GL_n(\mathcal{O})} ([R\psi_{\beta^0}]\otimes C_c^{\infty}(\GL_{n-k}(\mathcal{O})))\ ,
\]
where we take functions with values in $\bar{\mathbb{Q}}_{\ell}$.
\item[{\rm (vi)}] We have the following equality
\[
\phi_{\tau,h}(\beta) = \mathrm{tr}\left(h\times N\beta^{\mathrm{et}} | \mathrm{Ind}_{P_k(\mathcal{O})}^{\GL_n(\mathcal{O})} \phi_{k,\tau}(\beta^0)\otimes C_c^{\infty}(\GL_{n-k}(\mathcal{O})) \right)\ ,
\]
with $N\beta^{\mathrm{et}}$ acting through right multiplication by $(N\beta^{\mathrm{et}})^{-1}$. Here $\phi_{k,\tau}(\beta^0)$ is the distribution taking $h_k\in C_c^{\infty}(\GL_k(\mathcal{O}))$ to
\[
\phi_{\tau,h_k}(\beta^0)=\tr(\tau\times h_k^{\vee}|[R\psi_{\beta^0}])\ ,
\]
which we consider as an element in the Grothendieck group (with $\mathbb{C}$-coefficients) of admissible representations of $\GL_{n-k}(\mathcal{O})$.
\end{altenumerate}
\end{prop}

\begin{proof} This is basically well-known. Part (i) follows from the canonical decomposition
\[
\overline{H}_{\beta}=\overline{H}_{\beta}^0 \times \overline{H}_{\beta}^{\mathrm{et}}\ ,
\]
identifying sections of $\overline{H}_{\beta}$ with points of $\overline{H}_{\beta}^{\mathrm{et}}$.

It is clear that the Galois action on the connected components is through the action on $\overline{H}_{\beta}^{\mathrm{et}}[\varpi^m](\bar{k})$. This action does not change $\ker v$, hence the first part of (ii) follows. The second is obvious.

Part (iii) follows from Proposition 4.5 of \cite{Drinfeld}.

Now part (iv) follows from the invariance of nearby cycles under power series extensions, see e.g. \cite{HarrisTaylor}, Lemma I.5.6. Also part (v) is an immediate consequence of part (iv) and the facts about the $\GL_n(\mathcal{O}/\varpi^m)$-action from part (i) and (ii).

Finally, part (vi) follows from part (v) after taking dual $\GL_n(\mathcal{O})$-representations.
\end{proof}

\section{Descent properties of the test function: Harmonic analysis}

In this section, we will translate the descent formula established in the last section to establish condition (i) of Lemma \ref{LemmaToProve}. Basically, this is an application of the Weyl integration formula, turning a statement about the values of $\phi_{\tau,h}$ on elements into a statement about its traces on representations.

Let $P_k$ be the standard parabolic with Levi $\GL_k\times \GL_{n-k}$ and let $N_k$ be its unipotent radical. For any admissible representation $\pi$ of $\GL_n(F)$ of finite length, let $\pi_{N_k}$ be its (unnormalized) Jacquet module with respect to $N_k$. Assume that
\[
\pi_{N_k} = \sum_{i=1}^{t_{\pi,k}} \pi_{N_k,i}^1\otimes \pi_{N_k,i}^2
\]
as elements of the Grothendieck group of representations of $\GL_k(F)\times \GL_{n-k}(F)$. Let $\Theta_{\Pi}$ be the distribution on $\GL_n(F_r)$ given by $\Theta_{\Pi}(\beta)=\Theta_{\pi}(N\beta)$. Define $\Theta_{\Pi_{N_k,i}^1}$ analogously. For any $\phi\in C_c^{\infty}(\GL_n(F_r))$, we will write $\Theta_{\Pi}(\phi)$ as $\tr( (\phi,\sigma) |\Pi )$, thinking of $\Pi$ as the base-change lift of $\pi$.

\begin{lem}\label{WeylIntegration} In this situation,
\[
\tr ( (\phi_{\tau,h},\sigma)|\Pi ) = \sum_{k=1}^n\sum_{i=1}^{t_{\pi,k}} p^{(n-k) r} \tr\left( h | \mathrm{Ind}_{P_k(\mathcal{O})}^{\GL_n(\mathcal{O})} \left(\tr( (\phi_{k,\tau}\chi_{B_k},\sigma) | \Pi^1_{N_k,i})\right)\otimes \pi^2_{N_k,i} \right)\ ,
\]
where $\tr( (\phi_{k,\tau}\chi_{B_k},\sigma) | \Pi^1_{N_k,i})$ is the distribution on $\GL_k(\mathcal{O})$ sending $h_k$ to
\[
\tr( (\phi_{\tau,h_k^{\vee}}\chi_{B_k},\sigma) | \Pi^1_{N_k,i})\ ,
\]
where $\chi_{B_k}$ is the characteristic function of the set $B_k$ of all basic elements in
\[
\GL_k(\mathcal{O}_r)\mathrm{diag}(\varpi,1,\ldots,1)\GL_k(\mathcal{O}_r)\ .
\]
Again, we identify conjugation-invariant distributions with elements in the Grothendieck group of admissible representations to make sense of this formula.
\end{lem}

\begin{proof} We follow the proof of Lemma 5.5 in \cite{ScholzeGLn}. This gives
\[\begin{aligned}
&\tr ( (\phi_{\tau,h},\sigma)|\Pi ) = \sum_{k=1}^n p^{(n-k) r} \sum_{i=1}^{t_{\pi,k}} \sum_{\substack{T_k\subset \GL_k, T_{n-k}\subset \GL_{n-k}\\ T_k\ {\rm anisotropic}}} |W(T_k\times T_{n-k},\GL_k\times \GL_{n-k})|^{-1}\\
&\times \int \Delta_{\GL_k(F)}^2 (Nt_1)\Delta_{\GL_{n-k}(F)}^2 (Nt_2) TO_{t\sigma}^{\GL_k\times \GL_{n-k}}(\phi_{\tau,h}) \Theta_{\pi^1_{N_k,i}}(Nt_1) \Theta_{\pi^2_{N_k,i}}(Nt_2)dt_1 dt_2\ ,
\end{aligned}\]
where the integral runs over the set of all $t=(t_1,t_2)$ in
\[
T_k(F_r)^{1-\sigma}\times T_{n-k}(F_r)^{1-\sigma}\backslash T_k(F_r)_1\times T_{n-k}(\mathcal{O}_r)\ ,
\]
where $T_k(F_r)_1$ is the set of all $t_1\in T_k(F_r)$ with $v_F(\det t_1)=1$. Further,
\[
W(T_k\times T_{n-k},\GL_k\times \GL_{n-k})
\]
is the normalizer of $T_k(F)\times T_{n-k}(F)$ in $\GL_k(F)\times \GL_{n-k}(F)$ divided by $T_k(F)\times T_{n-k}(F)$.

Now, we keep $t_1$ fixed and look at the sum over $T_{n-k}$ combined with the integral over $t_2$:
\[\begin{aligned}
\sum_{T_{n-k}\subset \GL_{n-k}} &|W(T_{n-k},\GL_{n-k})|^{-1}\\
&\times \int\limits_{T_{n-k}(F_r)^{1-\sigma}\backslash T_{n-k}(\mathcal{O}_r)} \Delta_{\GL_{n-k}(F)}^2 (Nt_2) TO_{t_2\sigma}^{\GL_{n-k}} (\phi_{\tau,h}(t_1,\cdot)) \Theta_{\pi^2_{N_k,i}}(Nt_2) dt_2\ .
\end{aligned}\]
This is exactly the twisted Weyl integration formula computing $\tr( \phi_{\tau,h}(t_1,\cdot)|\Pi^2_{N_k,i})$.

\begin{lem} We have
\[
\tr( (\phi_{\tau,h}(t_1,\cdot),\sigma) | \Pi_{n-k} ) = \tr( h | \mathrm{Ind}_{P_k(\mathcal{O})}^{\GL_n(\mathcal{O})} \phi_{k,\tau}(t_1)\otimes \pi_{n-k} )
\]
for any irreducible smooth representation $\pi_{n-k}$ of $\GL_{n-k}(F)$.
\end{lem}

\begin{proof} For this, we first recall the following formula.

\begin{lem} Let $f$ be a $\GL_{n-k}(\mathcal{O})$-conjugation-invariant locally constant function on $\GL_{n-k}(\mathcal{O})$ and let $\phi\in C_c^{\infty}(\GL_{n-k}(\mathcal{O}_r))$ be defined by $\phi(\beta)=f(N\beta)$. Then
\[
\int_{\GL_{n-k}(\mathcal{O}_r)} \phi(\beta) d\beta = \int_{\GL_{n-k}(\mathcal{O})} f(\gamma) d\gamma\ .
\]
\end{lem}

\begin{proof} This is proved in exactly the same way as Corollary \ref{IntegrationFormulaEtale}, replacing Dieudonn\'{e} modules by relative Dieudonn\'{e} modules.
\end{proof}

Hence if $f_{\tau,h}(t_1,\cdot)\in C_c^{\infty}(\GL_{n-k}(\mathcal{O}))$ is defined by
\[
f_{\tau,h}(t_1,\gamma_2) = \mathrm{tr}\left(h\times \gamma_2 | \mathrm{Ind}_{P_k(\mathcal{O})}^{\GL_n(\mathcal{O})} \phi_{k,\tau}(t_1)\otimes C_c^{\infty}(\GL_{n-k}(\mathcal{O})) \right)\ ,
\]
then $\tr( (\phi_{\tau,h}(t_1,\cdot),\sigma) | \Pi_{n-k} ) = \tr( f_{\tau,h}(t_1,\cdot) | \pi_{n-k} )$ by Proposition \ref{GeometricReduction}, part (vi). We need to see that
\[
\tr( f_{\tau,h}(t_1,\cdot) | \pi_{n-k} ) = \tr( h | \mathrm{Ind}_{P_k(\mathcal{O})}^{\GL_n(\mathcal{O})} \phi_{k,\tau}(t_1)\otimes \pi_{n-k} )\ ,
\]
but this follows from writing
\[
C_c^{\infty}(\GL_{n-k}(\mathcal{O})) = \bigoplus_{\pi_{n-k}^{\prime}} \pi_{n-k}^{\prime}\otimes \pi_{n-k}^{\prime\vee}\ ,
\]
where $\pi_{n-k}^{\prime}$ runs over all irreducible representations of $\GL_{n-k}(\mathcal{O})$, and decomposing $\pi_{n-k}$ into irreducible representations.
\end{proof}

Now, we look at the sum over $T_k$ and the integral over $t_1$, which reads
\[
\sum_{\substack{T_k\subset \GL_k\\T_k\ {\rm anisotropic}}} |W(T_k,\GL_k)|^{-1}\int\limits_{T_k(F_r)^{1-\sigma}\backslash T_k(F_r)_1} \Delta_{\GL_k(F)}^2 (Nt_1) TO_{t_1\sigma}^{\GL_k}(\psi) \Theta_{\pi^1_{N_k,i}}(Nt_1)dt_1\ ,
\]
where $\psi(t_1) = \tr(h|\mathrm{Ind}_{P_k(\mathcal{O})}^{\GL_n(\mathcal{O})} \phi_{k,\tau}(t_1)\otimes \pi^2_{N_k,i} )$. This is the twisted Weyl integration formula again, this time for $\psi\chi_{B_k}$ and $\GL_k(F)$. It calculates the trace
\[
\tr((\psi\chi_{B_k},\sigma)|\Pi^1_{N_k,i})\ ,
\]
which is exactly
\[
\tr\left( h | \mathrm{Ind}_{P_k(\mathcal{O})}^{\GL_n(\mathcal{O})} (\tr( (\phi_{k,\tau}\chi_{B_k},\sigma) | \Pi^1_{N_k,i}))\otimes \pi^2_{N_k,i} \right)
\]
by tracing through the definitions.
\end{proof}

\begin{thm} Assume that Theorem \ref{MainTheorem2} holds true for all $n^{\prime}<n$. Assume that $\pi$ is the normalized parabolic induction of an irreducible representation $\pi_1\otimes \cdots \otimes \pi_t$ of $\GL_{n_1}(F)\times\cdots\times\GL_{n_t}(F)$, where $t\geq 2$: Hence there exist representations $\mathrm{rec}(\pi_1),\ldots,\mathrm{rec}(\pi_t)$ of $W_F$ such that
\[
\tr(f_{\tau,h_i}|\pi_i) = \tr(\tau|\mathrm{rec}(\pi_i)) \tr(h_i|\pi_i)\ ,
\]
for all $i$ and all functions $h_i\in C_c^{\infty}(\GL_{n_i}(\mathcal{O}))$. Then we have
\[
\tr(f_{\tau,h}|\pi) = \tr\left(\tau|\mathrm{rec}(\pi_1)(\tfrac{n-n_1}2)\oplus\ldots\oplus\mathrm{rec}(\pi_t)(\tfrac{n-n_t}2)\right) \tr(h|\pi)
\]
for all functions $h\in C_c^{\infty}(\GL_n(\mathcal{O}))$.
\end{thm}

\begin{proof} The proof reduces to Lemma \ref{WeylIntegration} as the proof of Lemma 6.8 in \cite{ScholzeGLn} reduces to Lemma 6.5 in \cite{ScholzeGLn}.
\end{proof}

This establishes condition (i) of Lemma \ref{LemmaToProve}.

\section{Comparison of two functions}\label{LocalComp}

In this section, we compare our function $\phi_{\tau,h}$ with a closely related function $\phi_{\tau,h,h^{\prime}}$ that arises naturally as a local ingredient in the global setup.

Fix a semisimple algebra $D^{\prime}$ over $\mathbb{Q}_p$ with maximal order $\mathcal{O}_{D^{\prime}}$. Then we define the algebra $D = M_n(F)\times D^{\prime}$ with maximal order $\mathcal{O}_D = M_n(\mathcal{O})\times \mathcal{O}_{D^{\prime}}$. Consider the group scheme $\mathbf{G}_D$ over $\mathbb{Z}_p$ given by
\[
\mathbf{G}_D(R) = (\mathcal{O}_D\otimes_{\mathbb{Z}_p} R)^{\times}\ ,
\]
which decomposes as $\mathbf{G}_D = \mathrm{Res}_{\mathcal{O}/\mathbb{Z}_p} \GL_n\times \mathbf{G}_{D^{\prime}}$, where
\[
\mathbf{G}_{D^{\prime}}(R) = (\mathcal{O}_{D^{\prime}}\otimes_{\mathbb{Z}_p} R)^{\times}\ .
\]
In this context, we need the following type of $p$-divisible groups with extra structure.

\begin{definition}\label{DefOBGroup} \begin{altenumerate}
\item[{\rm (i)}] Let $S$ be an $\mathcal{O}$-scheme. An $(\mathcal{O},D^{\prime})$-group over $S$ is a $p$-divisible group $\underline{H}$ with an action
\[
\iota: \mathcal{O}_D^{\mathrm{op}}\longrightarrow \mathrm{End}(\underline{H})
\]
such that $\underline{H}$ decomposes as $H^n\times H^{\prime}$ under the action of $\mathcal{O}_D^{\mathrm{op}}=M_n(\mathcal{O})^{\mathrm{op}}\times \mathcal{O}_{D^{\prime}}^{\mathrm{op}}$, where $H$ is a one-dimensional $\varpi$-divisible $\mathcal{O}$-module of height $n$ and $H^{\prime}$ is a $D^{\prime}$-group in the sense of Definition \ref{DefBGroup}.

\item[{\rm (ii)}] A level-$m$-structure on an $(\mathcal{O},D^{\prime})$-group $\underline{H}$ consists of a Drinfeld-level-$m$-structure on $H$ and an isomorphism
\[
\mathcal{O}_{D^{\prime}}^{\mathrm{op}}/p^m\cong H^{\prime}[p^m]
\]
of $\mathcal{O}_{D^{\prime}}^{\mathrm{op}}$-modules.
\end{altenumerate}
\end{definition}

There are two ways to parametrize $(\mathcal{O},D^{\prime})$-groups over $\kappa_r$. For the first, one looks at the usual Dieudonn\'{e} module $\underline{M}$ of $\underline{\overline{H}}$. Our assumptions imply that there is an isomorphism of $\mathcal{O}_D\otimes_{\mathbb{Z}_p} W(\kappa_r)$-modules
\[
\underline{M}\cong \mathcal{O}_D\otimes_{\mathbb{Z}_p} W(\kappa_r)\ .
\]
If we write $\sigma_0$ for the absolute Frobenius of $W(\kappa_r)$, then, writing the Frobenius $F$ of $\underline{M}$ as $F=\delta_0^{-1}\sigma_0$, we get an element
\[
\delta_0\in \mathbf{G}_D(W(\kappa_r)_{\mathbb{Q}})\ ,
\]
which is well-defined up to $\sigma_0$-conjugation by an element of $\mathbf{G}_D(W(\kappa_r))$. Here $W(\kappa_r)_{\mathbb{Q}}$ is the fraction field of $W(\kappa_r)$.

The second possibility is to look at the decomposition $\underline{\overline{H}}=\overline{H}^n\times \overline{H}^{\prime}$. As above, one can parametrize $\overline{H}$ by an element $\beta\in \GL_n(F_r)$, well-defined up to $\sigma$-conjugation by $\GL_n(\mathcal{O}_r)$. Further, $\overline{H}^{\prime}$ gives an element $\gamma^{\prime}\in \mathcal{O}_{D^{\prime}}^{\times}$ using the Galois parametrization, and an element $\beta^{\prime}\in (\mathcal{O}_{D^{\prime}}\otimes_{\mathbb{Z}_p} W(\kappa_r))^{\times}$ using the Dieudonn\'{e} parametrization. These satisfy $\gamma^{\prime}=N\beta^{\prime}$.

Going through all definitions shows the following proposition, giving the relationship between the two parametrizations.

\begin{prop} Under the canonical decomposition
\[\begin{aligned}
\mathbf{G}_D(W(\kappa_r)_{\mathbb{Q}})&\cong \Big(\prod_{\alpha: \kappa\longrightarrow \kappa_r} \GL_n(F_r)\Big)\times \mathbf{G}_{D^{\prime}}(W(\kappa_r)_{\mathbb{Q}})\\
&=\GL_n(F_r)\times \Big(\prod_{\alpha: \kappa\longrightarrow \kappa_r, \alpha\neq \mathrm{id}} \GL_n(F_r)\Big)\times \mathbf{G}_{D^{\prime}}(W(\kappa_r)_{\mathbb{Q}})\ ,
\end{aligned}\]
the element $\delta_0$ is $\mathbf{G}_D(W(\kappa_r))$-$\sigma_0$-conjugate to $((\beta^{-1})^t,1,\beta^{\prime})$.
\end{prop}

Now fix an $(\mathcal{O},D^{\prime})$-group $\underline{\overline{H}}$ over $\kappa_r$. We get a deformation space $R_{\underline{\overline{H}}}$ of $\underline{\overline{H}}$ with universal deformation $\underline{H}$, and a covering $R_{\underline{\overline{H}},m}/R_{\underline{\overline{H}}}$ parametrizing level-$m$-structures on the universal deformation.

\begin{prop}\begin{altenumerate}
\item[{\rm (i)}] The canonical map $R_{\underline{\overline{H}}}\longrightarrow R_{\beta}$ is an isomorphism.
\item[{\rm (ii)}] The canonical map
\[
R_{\underline{\overline{H}},m}\longrightarrow R_{\beta,m}\times_{R_{\beta}} (H^{\prime}[p^m])^{\times}
\]
is an isomorphism, where $(H^{\prime}[p^m])^{\times}\subset H^{\prime}[p^m]$ is the open and closed subset of $\mathcal{O}_{D^{\prime}}^{\mathrm{op}}$-generators of the $\mathcal{O}_{D^{\prime}}^{\mathrm{op}}$-module $H^{\prime}[p^m]$. The $\mathcal{O}_{D^{\prime}}^{\times}$-action on $R_{\underline{\overline{H}},m}$ is given by sending $b\in \mathcal{O}_{D^{\prime}}^{\times}$ to left multiplication by $b^{-1}$ on $H^{\prime}[p^m]$.
\end{altenumerate}
In particular, the covering $R_{\underline{\overline{H}},m}/R_{\underline{\overline{H}}}$ is a finite Galois cover with Galois group
\[
\GL_n(\mathcal{O}/\varpi^m)\times (\mathcal{O}_{D^{\prime}}/p^m)^{\times}\ .
\]
\end{prop}

\begin{proof} This follows from rigidity of \'{e}tale covers: There is a unique deformation of a $D^{\prime}$-group to any infinitesimal thickening.
\end{proof}

Again, we can consider the global sections of the nearby cycle sheaves
\[
R\psi_{\underline{\overline{H}}} = \mathop{\mathrm{lim}}\limits_{\longrightarrow} H^0(R\psi_{\mathrm{Spf}\ R_{\underline{\overline{H}},m}} \bar{\mathbb{Q}}_{\ell})\ .
\]
They carry an action of $W_{F_r}\times \mathbf{G}_D(\mathbb{Z}_p)$. We get the following corollary of the last proposition.

\begin{cor} There is a canonical $W_{F_r}\times \mathbf{G}_D(\mathbb{Z}_p)$-equivariant isomorphism
\[
R\psi_{\underline{\overline{H}}}\cong R\psi_{\beta}\otimes C_c^{\infty}(\mathcal{O}_{D^{\prime}}^{\times})\ ,
\]
where we take functions with values in $\bar{\mathbb{Q}}_{\ell}$.
\end{cor}

In particular, let $\tau\in W_F$ project to the $r$-th power of Frobenius, and let $h\in C_c^{\infty}(\GL_n(\mathcal{O}))$ and $h^{\prime}\in C_c^{\infty}(\mathcal{O}_{D^{\prime}}^{\times})$ have values in $\mathbb{Q}$. We define the function $\phi_{\tau,h,h^{\prime}}(\delta_0)$ by
\[
\phi_{\tau,h,h^{\prime}}(\delta_0) = \tr(\tau\times h^{\vee}\times h^{\prime}|R\psi_{\underline{\overline{H}}})\ ,
\]
if $\underline{\overline{H}}$ corresponds to $\delta_0$, and by $0$ if there is no such $\underline{\overline{H}}$. Then our results of Section \ref{DefSpaces} imply that $\phi_{\tau,h,h^{\prime}}\in C_c^{\infty}(\mathbf{G}_D(W(\kappa_r)_{\mathbb{Q}}))$, and is independent of $\ell$. In fact, we have the following result, which is the main result of this section.

\begin{lem}\label{LocalComparison} For any $\tau$, $h$, $h^{\prime}$, the functions $\phi_{\tau,h,h^{\prime}}$ and $f_{\tau,h}^{\vee}\times h^{\prime}$ are associated.
\end{lem}

\begin{proof} We may assume that $h^{\prime}$ is invariant under $\mathcal{O}_{D^{\prime}}^{\times}$-conjugation. Take $\underline{\overline{H}}$ which corresponds to $\delta_0$ and to $(\beta,\gamma^{\prime})$ in the two parametrizations. In that case, we claim that
\[
\tr(\tau\times h^{\prime}|C_c^{\infty}(\mathcal{O}_{D^{\prime}}^{\times})) = h^{\prime}(\gamma^{\prime})\ .
\]
Indeed, the action of $\tau$ on $H[p^m](\bar{\kappa})$ is through right multiplication with $\gamma^{\prime}$, and hence the left-hand side may be computed as
\[
\int_{\mathcal{O}_{D^{\prime}}^\times} h^{\prime}(b^{-1}\gamma^{\prime} b) db = h^{\prime}(\gamma^{\prime})\ ,
\]
by assumption on $h^{\prime}$.

It follows that
\[
\phi_{\tau,h,h^{\prime}}(\delta_0) = \phi_{\tau,h}(\beta)h^{\prime}(\gamma^{\prime})
\]
if $\delta_0$ and $(\beta,\gamma^{\prime})$ parametrize the same $(\mathcal{O},D^{\prime})$-group over $\kappa_r$. Now the same argument as in the proof of Proposition \ref{BGroupAssociatedFcts} shows the lemma.
\end{proof}

\section{Embedding into a Shimura variety}

In order to apply certain global arguments, we embed the situation into a Shimura variety. These Shimura varieties are of the kind considered by Harris-Taylor, and previously in a more general set-up by Kottwitz. We use Kottwitz' notation as in \cite{KottwitzPoints}, \cite{KottwitzLambdaAdic}, except that we stick to our convention of using bold-face letters for global objects.

We start with the following data:
\begin{altitemize}
\item a totally real field $\mathbb{F}_0$ of even degree over $\mathbb{Q}$;
\item an infinite place $\tau$ of $\mathbb{F}_0$;
\item a finite place $x_0$ of $\mathbb{F}_0$;
\item an imaginary quadratic field $\mathbb{K}$ such that the rational prime below $x_0$ splits in $\mathbb{K}$;
\item an embedding $\mathbb{K}\hookrightarrow \mathbb{C}$;
\item a positive integer $n$.
\end{altitemize}
We get the CM field $\mathbb{F}=\mathbb{K}\mathbb{F}_0$, and two places $x,x^c$ of $\mathbb{F}$ over $\mathbb{F}_0$.

\begin{lem} Given these data, there exist a central division algebra $\mathbb{D}$ over $\mathbb{F}$ of dimension $n^2$, an involution $\ast$ on $\mathbb{D}$ of the second kind, and a homomorphism
\[
h_0:\mathbb{C}\longrightarrow \mathbb{D}_{\mathbb{R}}
\]
such that $x\longmapsto h_0(i)^{-1}x^{\ast}h_0(i)$ is a positive involution on $\mathbb{D}_{\mathbb{R}}$ and such that the following properties hold:
\begin{altenumerate}
\item[{\rm (i)}] The algebra $\mathbb{D}$ splits at all places of $\mathbb{F}$ different from $x,x^c$;
\item[{\rm (ii)}] If $\mathbf{G}_0/\mathbb{F}_0$ is the algebraic group defined by
\[
\mathbf{G}_0(R) = \{ g\in (\mathbb{D}\otimes_{\mathbb{F}_0} R)^{\times}\mid gg^{\ast} = 1\}\ ,
\]
then $\mathbf{G}_0$ is quasisplit at all finite places of $\mathbb{F}_0$ that do not split in $\mathbb{F}$, is a unitary group of signature $(1,n-1)$ at $\tau$, and is a unitary group of signature $(0,n)$ at all other infinite places\footnote{With the convention on the signature explained in the next few lines.}.
\end{altenumerate}
\end{lem}

\begin{rem} The order in the signature is important: Any infinite place $\tau^{\prime}$ of $\mathbb{F}_0$ induces an isomorphism $\mathbb{D}\otimes_{\mathbb{F}_0}\mathbb{R}\cong M_n(\mathbb{C})$, unique up to conjugation, normalized by our choice of embedding $\mathbb{K}\hookrightarrow \mathbb{C}$. Under this identification, $h_0$ takes the form
\[
h_0(z) = (\mathrm{diag}(z,\ldots,z,\overline{z},\ldots,\overline{z}))_{\tau^{\prime}}\in \mathbb{D}_{\mathbb{R}} = \prod_{\tau^{\prime}} \mathbb{D}\otimes_{\mathbb{F}_0}\mathbb{R}\cong \prod_{\tau^{\prime}} M_n(\mathbb{C})\ .
\]
We say that $\mathbf{G}_0$ has signature $(p,q)$ at $\tau^{\prime}$ if the number of $z$'s appearing at $\tau^{\prime}$ is $p$, and the number of $\overline{z}$'s is $q$. It is then easy to see that indeed $\mathbf{G}_0\otimes_{\mathbb{F}_0} \mathbb{R}$ is isomorphic to the unitary group $U(p,q)$ of signature $(p,q)$.
\end{rem}

\begin{proof} This is contained in Section I.7 of \cite{HarrisTaylor}, in particular Lemma I.7.1. Note that our $\mathbb{D}$ corresponds to $B^{\mathrm{op}}$ and our $\ast$ corresponds to $\sharp_\beta$ in the notation of \cite{HarrisTaylor}.
\end{proof}

We get the reductive group $\mathbf{G}/\mathbb{Q}$ defined by
\[
\mathbf{G}(R) = \{g\in (\mathbb{D}\otimes R)^{\times}\mid gg^{\ast}\in R^{\times}\}\ .
\]

\begin{rem} An assumption about the ramification of $\mathbb{D}$ is necessary to apply the global Jacquet-Langlands correspondence relating automorphic representations of $\mathbb{D}^{\times}$ to automorphic representations of $\GL_n/\mathbb{F}$. The assumption that $\mathbb{F}$ be the composite of a totally real field and an imaginary quadratic field implies that the action of the Galois group on the dual group $\hat{\mathbf{G}}$ is through the quotient $\mathrm{Gal}(\mathbb{K}/\mathbb{Q})$, which makes applications of base-change possible. The restriction on the degree of $\mathbb{F}_0$ over $\mathbb{Q}$ is needed to make the lemma true.
\end{rem}

Restricting $h_0$ to $\mathbb{C}^{\times}$ and considering it as a morphism of algebraic groups over $\mathbb{R}$ gives a map
\[
h: \mathrm{Res}_{\mathbb{C}/\mathbb{R}} \mathbb{G}_m\longrightarrow \mathbf{G}_{\mathbb{R}}\ .
\]
Then the datum $(\mathbf{G}, h^{-1})$ defines a projective Shimura variety $\mathrm{Sh}_K$ for any compact open subgroup $K\subset \mathbf{G}(\mathbb{A}_f)$, cf. \cite{KottwitzPoints}.

Over $\mathbb{C}$, the group $\mathbf{G}$ decomposes as
\[
\mathbf{G}_\mathbb{C} = \Big(\prod\limits_{\mathbb{F}_0\hookrightarrow \mathbb{C}} \GL_n\Big)\times \mathbb{G}_m\ ,
\]
where the projections to $\GL_n$ are given by the projection of
\[
\mathbb{D}\otimes_\mathbb{Q} \mathbb{C} = \prod\limits_{\mathbb{F}\hookrightarrow \mathbb{C}} M_n(\mathbb{C})
\]
to the factor corresponding to the embedding $\mathbb{F}_0\hookrightarrow \mathbb{C}$ and the fixed embedding $\mathbb{K}\hookrightarrow \mathbb{C}$. Further, the projection to $\mathbb{G}_m$ is the natural map $g\longmapsto gg^{\ast}$. Recall the morphism
\[\begin{aligned}
\mathbb{G}_m&\longrightarrow \mathbb{G}_m\times \mathbb{G}_m \cong (\mathrm{Res}_{\mathbb{C}/\mathbb{R}} \mathbb{G}_m)\otimes_{\mathbb{R}} \mathbb{C}\\
z&\longmapsto (z,1)
\end{aligned}\]
where the first factor corresponds to the identity map $\mathbb{C}\longrightarrow \mathbb{C}$ and the second factor corresponds to complex conjugation. Precomposing $h$ with this morphism defines a (minuscule) cocharacter
\[
\mu:\mathbb{G}_m\longrightarrow \mathbf{G}_\mathbb{C}\ .
\]
Our assumptions imply that it is given by $z\mapsto \mathrm{diag}(z,1,\ldots,1)$ on the factor corresponding to $\tau$, by $z\mapsto 1$ on the other $\GL_n$-factors, and by $z\mapsto z$ on the $\mathbb{G}_m$-factor.

From here, it follows that the reflex field $E$ is canonically isomorphic to $\mathbb{F}$: The infinite place $\tau$ of $\mathbb{F}_0$ and the embedding $\mathbb{K}\hookrightarrow \mathbb{C}$ give a canonical embedding of $\mathbb{F}$ into $\mathbb{C}$, which identifies $\mathbb{F}$ with $E\subset \mathbb{C}$.

Our assumptions on $p$ are the following. We require that $p$ splits in $\mathbb{K}$. Further, we fix a place $w$ of $\mathbb{F}$ above $p$, inducing a place $u$ of $\mathbb{K}$ above $p$, and require that $\mathbb{D}$ splits at $w$, i.e. $w\neq x,x^c$. We set $F=\mathbb{F}_w$ and use the local notation $\mathcal{O}$, $\kappa$, etc., as before. We choose some rational prime $\ell\neq p$ and an isomorphism $\bar{\mathbb{Q}}_{\ell}\cong \mathbb{C}$. Note that the Grothendieck groups of continuous representations of $W_F$ with coefficients in either $\bar{\mathbb{Q}}_{\ell}$ or $\mathbb{C}$ get identified, and we will ignore the distinction.

Now let $\xi$ be an irreducible algebraic representation of $\mathbf{G}$. We get $\ell$-adic local systems $\mathcal{F}_K$ on $\mathrm{Sh}_K$ for all compact open subgroups $K\subset \mathbf{G}(\mathbb{A}_f)$, to which the action of $\mathbf{G}(\mathbb{A}_f)$ extends, cf. e.g. \cite{KottwitzPoints}. Then we consider the direct limit of the \'{e}tale cohomology groups of $\mathrm{Sh}_K$ with coefficients in $\mathcal{F}_K$, i.e.
\[
H_{\xi}^{\ast} = \mathop{\mathrm{lim}}\limits_{\mathop{\longrightarrow}\limits_K} H^{\ast}(\mathrm{Sh}_K,\mathcal{F}_K)\ .
\]
We also consider the alternating sum $[H_{\xi}]$ as an element in the Grothendieck group of $\ell$-adic representations of $\mathrm{Gal}(\overline{\mathbb{F}}/\mathbb{F})\times \mathbf{G}(\mathbb{A}_f)$ with continuous Galois action and admissible smooth $\mathbf{G}(\mathbb{A}_f)$-action.

In order to apply the method of counting points modulo $p$, we first have to construct integral models of the Shimura varieties over $\mathcal{O}$. We have to explain how to choose the data of PEL type in \cite{KottwitzPoints}. We take the simple $\mathbb{Q}$-algebra $B=\mathbb{D}^{\mathrm{op}}$ and $V=D$ as a left $B$-module via right multiplication. This gives $C=\mathrm{End}_B(V)=\mathbb{D}$. Furthermore, choose $\xi\in \mathbb{D}^{\times}$ with $\xi^{\ast}=-\xi$ close to $h_0(i)$. Then $x\longmapsto \xi x^{\ast} \xi^{-1}$ defines a positive involution $\ast_B$ on $B$. Moreover, we have an alternating pairing $(\cdot,\cdot):V\times V\longrightarrow \mathbb{Q}$ defined by
\[
(x,y) = \mathrm{tr}_{\mathbb{F}/\mathbb{Q}}\mathrm{tr}_{\mathbb{D}/\mathbb{F}} (x\xi y^{\ast})\ ,
\]
where $\mathrm{tr}_{\mathbb{D}/\mathbb{F}}$ is the reduced trace. This pairing is compatible with $\ast_B$, i.e. $(bx,y)=(x,b^{\ast_B}y)$ for all $b\in B$. Finally, we are given the homomorphism
\[
h_0: \mathbb{C}\longrightarrow C_{\mathbb{R}} = \mathbb{D}_{\mathbb{R}}\ ,
\]
and if $\xi$ was chosen correctly, the form $(\cdot,h_0(i)\cdot)$ on $V_{\mathbb{R}}$ is positive definite. This gives the rational data of \cite{KottwitzPoints}.

We have a decomposition
\[
\mathbb{D}\otimes_{\mathbb{Q}} \mathbb{Q}_p\cong \prod_{\mathfrak{p}|p} \mathbb{D}\otimes_{\mathbb{F}} \mathbb{F}_{\mathfrak{p}}\ ,
\]
and for all places $\mathfrak{p}|u$ we fix a maximal order $\mathcal{O}_{\mathbb{D},\mathfrak{p}}$ in $\mathbb{D}\otimes_{\mathbb{F}} \mathbb{F}_{\mathfrak{p}}$. Equivalently, we have maximal orders $\mathcal{O}_{\mathbb{D},\mathfrak{p}}^{\mathrm{op}}$ in $B_{\mathfrak{p}}$. Using the involution $\ast_B$, these give rise to maximal orders in $B_{\mathfrak{p}}$ for all places $\mathfrak{p}$ of $\mathbb{F}$ above $p$, and hence to a maximal $\mathbb{Z}_{(p)}$-order $\mathcal{O}_B$ in $B$. We also get a unique self-dual $\mathbb{Z}_p$-lattice $\Lambda$ in $V\otimes \mathbb{Q}_p$ which is equal to $\mathcal{O}_{\mathbb{D},\mathfrak{p}}$ at all places $\mathfrak{p}$ dividing $u$.

Further, writing
\[
\mathcal{O}_D = \prod_{\mathfrak{p}|u} \mathcal{O}_{\mathbb{D},\mathfrak{p}}\ ,
\]
we get $\mathcal{O}_B\otimes \mathbb{Z}_p= \mathcal{O}_D^{\mathrm{op}}\times \mathcal{O}_D$. Moreover, $\mathcal{O}_{\mathbb{D},w}\cong M_n(\mathcal{O})$, giving
\[
\mathcal{O}_D = M_n(\mathcal{O})\times \mathcal{O}_{D^{\prime}}
\]
with the obvious definition of $\mathcal{O}_{D^{\prime}}$, placing us in the situation of Section \ref{LocalComp}. With that notation, we have
\[\begin{aligned}
\mathbf{G}_{\mathbb{Q}_p} &= (\mathbf{G}_D\otimes_{\mathbb{Z}_p} \mathbb{Q}_p)\times \mathbb{G}_m\\
&=\mathrm{Res}_{F/\mathbb{Q}_p} \GL_n\times (\mathbf{G}_{D^{\prime}}\otimes_{\mathbb{Z}_p} \mathbb{Q}_p)\times \mathbb{G}_m\ ,
\end{aligned}\]
giving us an integral model of $\mathbf{G}$ over $\mathbb{Z}_{(p)}$. In particular, we get a maximal compact open subgroup $K_p^0\subset \mathbf{G}(\mathbb{Q}_p)$, which decomposes as
\[
K_p^0 = \GL_n(\mathcal{O})\times \mathcal{O}_{D^{\prime}}^{\times}\times \mathbb{Z}_p^{\times}\ .
\]
For $m\geq 1$, we also consider the congruence subgroups
\[
K_p^m = (1+\varpi^mM_n(\mathcal{O}))\times (1+p^m\mathcal{O}_{D^{\prime}})\times \mathbb{Z}_p^{\times}\ .
\]

Assume we are given an abelian variety $A$ up to prime-to-$p$-isogeny together with a polarization $\lambda$ of degree prime to $p$ and a $\ast_B$-homomorphism $\iota:\mathcal{O}_B\longrightarrow \mathrm{End}(A)$ over a scheme on which $p$ is locally nilpotent. Looking at its $p$-divisible group $A[p^{\infty}]$, we get a decomposition
\[
A[p^{\infty}] = (H(A)^n\times H(A)^{\prime})\times (H(A)^n\times H(A)^{\prime})^{\vee}
\]
corresponding to
\[
\mathcal{O}_B\otimes\mathbb{Z}_p\cong (M_n(\mathcal{O})^{\mathrm{op}}\times \mathcal{O}_{D^{\prime}}^{\mathrm{op}})\times (M_n(\mathcal{O})\times \mathcal{O}_{D^{\prime}})\ .
\]
In the situations we will consider, the $p$-divisible group $\underline{H}(A) = H(A)^n\times H(A)^{\prime}$ with its $\mathcal{O}_D^{\mathrm{op}}$-action is an $(\mathcal{O},D^{\prime})$-group (so that the notion of level-$m$-structure from Definition \ref{DefOBGroup} applies). In fact, over schemes on which $p$ is locally nilpotent, this condition is equivalent to the determinant condition of Kottwitz, cf. \cite{KottwitzPoints} and Lemma III.1.2 of \cite{HarrisTaylor}.

Consider the functor that associates to a locally noetherian scheme $S$ over $\mathcal{O}$ the set of isomorphism classes of quadruples $(A,\lambda,\iota,\overline{\eta}^p,X_1,\eta)$ consisting of
\begin{altitemize}
\item a projective abelian scheme $A$ over $S$ up to prime-to-$p$-isogeny,
\item a polarization $\lambda: A\longrightarrow A^{\vee}$ of degree prime to $p$,
\item a $\ast_B$-homomorphism $\iota:\mathcal{O}_B\longrightarrow \mathrm{End}(A)$, satisfying the determinant condition,
\item a level-structure $\overline{\eta}^p$ away from $p$ of type $K^p$,
\item a level-$m$-structure $\eta_p$ on $\underline{H}(A)$,
\end{altitemize}
where two objects
\[
(A,\lambda,\iota,\overline{\eta}^p,\eta_p)\ ,\ (A^{\prime},\lambda^{\prime},\iota^{\prime},\overline{\eta}^{p\prime},\eta_p^{\prime})
\]
are called isomorphic if there exists an $\mathcal{O}_B$-linear prime-to-$p$-isogeny from $A$ to $A^{\prime}$ carrying $\lambda$ into a $\mathbb{Z}_{(p)}^{\times}$-multiple of $\lambda^{\prime}$, $\overline{\eta}^{p}$ into $\overline{\eta}^{p\prime}$ and $\eta_p$ into $\eta_p^{\prime}$.

For an explanation of the notion of a level-structure of type $K^p$, we refer to \cite{KottwitzPoints}, Section 5.

As in \cite{KottwitzPoints} and \cite{HarrisTaylor}, Section III.4, one sees that if $K^p$ is sufficiently small, then this moduli problem is represented by a projective scheme $\mathrm{Sh}_{K^p,m}$. Furthermore, the arguments of Kottwitz, cf. also \cite{HarrisTaylor}, sections III.1 and III.4, show that the generic fibre of $\mathrm{Sh}_{K^p,m}$ is the disjoint union of $|\mathrm{ker}^1(\mathbb{Q},\mathbf{G})|$ copies of the canonical model $\mathrm{Sh}_K$ of the Shimura variety associated to $(\mathbf{G},h^{-1},K)$, base-changed from $\mathbb{F}$ to $F=\mathbb{F}_w$, where $K=K_p^mK^p$.

There is an obvious action of $\mathrm{GL}_n(\mathcal{O})\times \mathcal{O}_{D^{\prime}}^{\times} \times \mathbf{G}(\mathbb{A}_f^p)$ on the tower of these integral models. Further, the algebraic representation $\xi$ gives rise to smooth $\ell$-adic sheaves $\mathcal{F}_{K^p,m}$ on $\mathrm{Sh}_{K^p,m}$ to which this action extends naturally. They are compatible with the sheaves on the Shimura variety $\mathrm{Sh}_K$.

\section{Counting points modulo $p$}

The main result of this section calculates the trace of certain operators on the cohomology of the Shimura varieties introduced in the last section.

Let $\tau\in \mathrm{Frob}^r I_F\subset W_F$, and take $h\in C_c^{\infty}(\GL_n(\mathcal{O}))$ and $h^{\prime}\in C_c^{\infty}(\mathcal{O}_{D^{\prime}}^{\times})$ with values in $\mathbb{Q}$. Further, fix the characteristic function $e_{\mathbb{Z}_p^{\times}}$ of $\mathbb{Z}_p^{\times}\subset \mathbb{Q}_p^{\times}$ and take $f^p\in C_c^{\infty}(\mathbf{G}(\mathbb{A}_f^p))$, again with values in $\mathbb{Q}$. Then we get
\[
f=h^{\vee}\times h^{\prime}\times e_{\mathbb{Z}_p^{\times}}\times f^p\in C_c^{\infty}(\mathbf{G}(\mathbb{A}_f))\ .
\]
We will compute the trace $\tr(\tau\times f|[H_{\xi}])$. We fix $m\geq 1$ such that $h^{\vee}\times h^{\prime}\times e_{\mathbb{Z}_p^{\times}}$ is bi-$K_p^m$-invariant.

Fix a sufficiently small compact open subgroup $K^p\subset \mathbf{G}(\mathbb{A}_f^p)$ such that $f^p$ is bi-$K^p$-invariant. In fact, assume that $f^p$ is the characteristic function of $K^pg^pK^p$ divided by the volume of $K^p$ for some $g^p\in \mathbf{G}(\mathbb{A}_f^p)$. We have the following diagram.
\[\xymatrix{
& \mathrm{Sh}_{K^p\cap (g^p)^{-1}K^pg^p,m}\ar[ld]^{\tilde{p}_1} \ar[d]^{\pi} \ar[rd]_{\tilde{p}_2} & \\
\mathrm{Sh}_{K^p,m}\ar[d]^{\pi} & \mathrm{Sh}_{K^p\cap (g^p)^{-1}K^pg^p,0}\ar[ld]^{p_1} \ar[rd]_{p_2} & \mathrm{Sh}_{K^p,m}\ar[d]^{\pi} \\
\mathrm{Sh}_{K^p,0} & & \mathrm{Sh}_{K^p,0}
}\]

We will evaluate $\mathrm{tr}(\tau\times f|[H_{\xi}])$ via the Lefschetz trace formula. Recall that the upper correspondence in the diagram above extends canonically to a correspondence $u: \tilde{p}_{2!}\tilde{p}_1^{\ast}\mathcal{F}_{K^p,m}\longrightarrow \mathcal{F}_{K^p,m}$ and let
\[
[f^p]:  H^\ast(\mathrm{Sh}_{K^p,m}\otimes \bar{F},\mathcal{F}_{K^p,m})\longrightarrow H^\ast(\mathrm{Sh}_{K^p,m}\otimes \bar{F},\mathcal{F}_{K^p,m})
\]
be the associated map on cohomology. Of course, $\tau$, $h$ and $h^{\prime}$ also act on the cohomology, and it is a standard fact that
\[
|\mathrm{ker}^1(\mathbb{Q},\mathbf{G})| \mathrm{tr}(\tau\times f|[H_{\xi}]) = \tr(\tau\times h^{\vee}\times h^{\prime}\times [f^p]|H^\ast(\mathrm{Sh}_{K^p,m}\otimes \bar{F},\mathcal{F}_{K^p,m}))\ .
\]

We use proper base change to rewrite the cohomology as
\[
H^\ast(\mathrm{Sh}_{K^p,m}\otimes \bar{F},\mathcal{F}_{K^p,m}) = H^\ast(\mathrm{Sh}_{K^p,0}\otimes \bar{\kappa},\pi_{\ast}R\psi\mathcal{F}_{K^p,m})\ .
\]
Let $r^{\prime}=r[\kappa:\mathbb{F}_p]$ and let $F^{r^{\prime}}$ be the $r^{\prime}$-th power of the relative Frobenius correspondence on $\mathrm{Sh}_{K^p,0}\otimes \bar{\kappa}$. Then the action of $W_F$ on the nearby cycle sheaves gives a correspondence
\[
F^{r^{\prime}\ast} \pi_{\ast}R\psi\mathcal{F}_{K^p,m}\buildrel \tau\over\longrightarrow \pi_{\ast}R\psi\mathcal{F}_{K^p,m}\ ,
\]
which on cohomology realizes the map $\tau$. Furthermore, the actions of $\GL_n(\mathcal{O})$ and $\mathcal{O}_{D^{\prime}}^{\times}$ give a correspondence
\[
\pi_{\ast}R\psi\mathcal{F}_{K^p,m}\buildrel {h^{\vee}\times h^{\prime}}\over\longrightarrow \pi_{\ast}R\psi\mathcal{F}_{K^p,m}
\]
which on cohomology realizes the map $h^{\vee}\times h^{\prime}$. Hence the composition $c$ of the two correspondences,
\[
c: \mathrm{Sh}_{K^p,0}\otimes \bar{\kappa}\buildrel p_1\circ F^{r^{\prime}}\over\longleftarrow \mathrm{Sh}_{K^p\cap (g^p)^{-1}K^pg^p,0}\otimes \bar{\kappa}\buildrel p_2\over\longrightarrow \mathrm{Sh}_{K^p,0}\otimes \bar{\kappa}
\]
with the corresponding composition of the correspondences on the sheaf $\pi_{\ast}R\psi\mathcal{F}_{K^p,m}$, realizes a map on the cohomology whose trace equals $\mathrm{tr}(\tau\times f|[H_{\xi}])$.

Now we use the Lefschetz trace formula, see e.g. \cite{Varshavsky}, Theorem 2.3.2. Hence we need to evaluate the traces $\mathrm{tr}_y(\pi_{\ast}R\psi\mathcal{F}_{K^p,m})$ at the fixed points $y$, i.e. $y\in \mathrm{Sh}_{K^p\cap (g^p)^{-1}K^pg^p,0}(\bar{\kappa})$ with $x=(p_1\circ F^{r^{\prime}})(y) = p_2(y)$.

At this point, we stop to check that all of Kottwitz' arguments concerning the parametrization of the fixed points $y$ go through in our situation. We just make some remarks on the crucial points.

\begin{altitemize}
\item The analogue of Lemma 7.2 in \cite{KottwitzPoints} is true; one can even require that the hermitian form be preserved on the nose. Indeed, our assumption that $p$ be split in $\mathbb{K}$ immediately reduces everything to the linear case, which is trivial to handle.
\end{altitemize}

In particular, as in the article of Kottwitz, \cite{KottwitzPoints}, p. 419, cf. also p. 429, we can associate to any fixed point $y$ an element $\delta\in \mathbf{G}(W(\kappa_r)_{\mathbb{Q}})$, well-defined up to $\sigma_0$-conjugation by $\mathbf{G}(W(\kappa_r))$.\footnote{Note that the number $r$ is denoted $j$ by Kottwitz.} In fact, note that the fixed point $y$ gives rise to an $(\mathcal{O},D^{\prime})$-group $\underline{\overline{H}}_y$ over $\kappa_r$. Then the element $\delta$ can be more explicitly written as
\[
\delta=(\delta_0,t)\in \mathbf{G}_D(W(\kappa_r)_{\mathbb{Q}})\times W(\kappa_r)_{\mathbb{Q}}^{\times} = \mathbf{G}(W(\kappa_r)_{\mathbb{Q}})\ ,
\]
where $\delta_0$ is the Dieudonn\'{e} parametrization of $\underline{\overline{H}}_y$ in the sense of Section \ref{LocalComp}, and $t$ is some element of $p$-adic valuation $-1$.

\begin{altitemize}
\item Lemma 14.1 in \cite{KottwitzPoints} stays true. The crucial point is to check that for all $\delta\in \mathbf{G}(W(\kappa_r)_{\mathbb{Q}})$ associated to a fixed point, the norm $N\delta$ can be represented by an element of $\mathbf{G}(\mathbb{Q}_p)$. This follows from our arguments in Section \ref{LocalComp}.
\item One can describe the double coset $K_p^0\delta K_p^0$ similarly to the description made on pages 430-431 in \cite{KottwitzPoints}, cf. Section \ref{LocalComp} again.
\end{altitemize}

As in Section 14 of \cite{KottwitzPoints}, we also get an element $\gamma\in \mathbf{G}(\mathbb{A}_f^p)$, well-defined up to conjugation, and an element $\gamma_0\in \mathbf{G}(\mathbb{Q})$, well-defined up to stable conjugation, which is stably conjugate to $\gamma$ and $N\delta$.

\begin{lem} We have
\[
\pi_{\ast}R\psi\mathcal{F}_{K^p,m}\cong (\pi_{\ast}R\psi\bar{\mathbb{Q}}_{\ell})\otimes \mathcal{F}_{K^p,0}\ .
\]
\end{lem}

\begin{proof} In the following calculation, we use several times that $\pi_{\ast}$ and $R\psi$ commute. First,
\[
\pi_{\ast}R\psi\mathcal{F}_{K^p,m}\cong R\psi\pi_{\ast}\pi^{\ast}\mathcal{F}_{K^p,0}\cong R\psi(\pi_{\ast}\bar{\mathbb{Q}}_{\ell}\otimes \mathcal{F}_{K^p,0})\ ,
\]
by construction of the sheaves and the projection formula.

Now we use that $\mathcal{F}_{K^p,0}$ extends to a smooth sheaf on $\mathrm{Sh}_{K^p,0}$, because of the construction of $\mathcal{F}_{K^p,0}$ using the tower of Shimura varieties with varying $K_{\ell}\subset \mathbf{G}(\mathbb{Q}_{\ell})$, which extend to the integral models considered. In particular, there is a canonical adjointness morphism between
\[
R\psi(\pi_{\ast}\bar{\mathbb{Q}}_{\ell}\otimes \mathcal{F}_{K^p,0})
\]
and
\[
(R\psi\pi_{\ast}\bar{\mathbb{Q}}_{\ell})\otimes \mathcal{F}_{K^p,0}\ .
\]
In order to check that this is an isomorphism we reduce to torsion coefficients and then localize in the \'{e}tale topology so that (the torsion version of) $\mathcal{F}_{K^p,0}$ becomes trivial. Then it is clear.
\end{proof}

This gives
\[
\mathrm{tr}_y(\pi_{\ast}R\psi\mathcal{F}_{K^p,m}) = \mathrm{tr}(\gamma_0|\xi) \mathrm{tr}_y(\pi_{\ast}R\psi\bar{\mathbb{Q}}_{\ell})\ ,
\]
cf. pages 433-434 in \cite{KottwitzPoints} for the computation of the trace on $\mathcal{F}_{K^p,0}$ which gives $\mathrm{tr}(\gamma_0|\xi)$.

Further, we have the following description of the completed local rings $\hat{\mathcal{O}}_{\mathrm{Sh}_{K^p,0},x}$. As before, $\mathcal{O}_r$ is the unramified extension of $\mathcal{O}$ of degree $r$ and $\breve{\mathcal{O}}$ is the completion of the maximal unramified extension of $\mathcal{O}$.

\begin{lem}\label{SerreTate}\begin{altenumerate}
\item[{\rm (i)}] There is a canonical isomorphism
\[
\hat{\mathcal{O}}_{\mathrm{Sh}_{K^p,0},x}\cong R_{\underline{\overline{H}}_y}\otimes_{\mathcal{O}_r} \breve{\mathcal{O}}\ ,
\]
which maps the correspondence
\[\xymatrix@C=-10pt{
&\mathrm{Spf}\ \hat{\mathcal{O}}_{\mathrm{Sh}_{K^p\cap (g^p)^{-1}K^pg^p,0},y}\ar[dl]_{p_1\circ \mathrm{Frob}^r}\ar[dr]^{p_2} & \\
\mathrm{Spf}\ \hat{\mathcal{O}}_{\mathrm{Sh}_{K^p,0},x} && \mathrm{Spf}\ \hat{\mathcal{O}}_{\mathrm{Sh}_{K^p,0},x}
}\]
to the Frobenius correspondence.

\item[{\rm (ii)}] The isomorphism in {\rm (i)} extends $\GL_n(\mathcal{O})\times \mathcal{O}_{D^{\prime}}^{\times}$-equivariantly to
\[
\pi^{-1}(\mathrm{Spf}\ \hat{\mathcal{O}}_{\mathrm{Sh}_{K^p,0},x})\cong \mathrm{Spf}\ R_{\underline{\overline{H}}_y,m}\otimes_{\mathcal{O}_r} \breve{\mathcal{O}}\ .
\]
The correspondence given by
\[\xymatrix@C=-10pt{
&\pi^{-1}(\mathrm{Spf}\ \hat{\mathcal{O}}_{\mathrm{Sh}_{K^p\cap (g^p)^{-1}K^pg^p,0},y})\ar[dl]_{\tilde{p}_1\circ \mathrm{Frob}^r}\ar[dr]^{\tilde{p}_2} & \\
\pi^{-1}(\mathrm{Spf}\ \hat{\mathcal{O}}_{\mathrm{Sh}_{K^p,0},x}) && \pi^{-1}(\mathrm{Spf}\ \hat{\mathcal{O}}_{\mathrm{Sh}_{K^p,0},x})
}\]
is mapped to the Frobenius correspondence under this identification.
\end{altenumerate}
\end{lem}

\begin{proof} This is an immediate consequence of the Serre-Tate theorem that deforming an abelian variety is equivalent to deforming its $p$-divisible group.
\end{proof}

Recall that the local term $\mathrm{tr}_y(\pi_{\ast}R\psi\bar{\mathbb{Q}}_{\ell})$ is the trace of the map
\[
\tau\times h^{\vee}\times h^{\prime}: (\pi_{\ast}R\psi\bar{\mathbb{Q}}_{\ell})_x = ((p_1\circ F^{r^{\prime}})^{\ast}\pi_{\ast}R\psi\bar{\mathbb{Q}}_{\ell})_y\longrightarrow (p_2^{\ast}\pi_{\ast}R\psi\bar{\mathbb{Q}}_{\ell})_y=(\pi_{\ast}R\psi\bar{\mathbb{Q}}_{\ell})_x\ .
\]
But Lemma \ref{SerreTate} and Theorem 3.1 of \cite{Berkovich2} imply that one can identify
\[
(\pi_{\ast}R\psi\bar{\mathbb{Q}}_{\ell})_x = H^0(R\psi_{\mathrm{Spf}\ R_{\underline{\overline{H}}_y,m}} \bar{\mathbb{Q}}_{\ell})\ ,
\]
equivariantly for the $\GL_n(\mathcal{O})\times \mathcal{O}_{D^{\prime}}^{\times}$-action and compatible with the action of $\tau$. Hence we see that by definition of $\phi_{\tau,h,h^{\prime}}(\delta_0)$, we have
\[
\mathrm{tr}_y(\pi_{\ast}R\psi\bar{\mathbb{Q}}_{\ell}) = \phi_{\tau,h,h^{\prime}}(\delta_0)\ .
\]
Now all the further discussion of \cite{KottwitzPoints} applies and proves the following theorem.

\begin{thm}\label{TraceHeckeOperatorTO} With the notation of \cite{KottwitzPoints}, we have
\[\begin{aligned}
\tr(\tau\times h^{\vee}\times h^{\prime}&\times e_{\mathbb{Z}_p^{\times}}\times f^p|[H_{\xi}])\\
&= \sum_{(\gamma_0;\gamma,\delta)} c(\gamma_0;\gamma,\delta) O_{\gamma}(f^p) TO_{\delta\sigma_0} (\phi_{\tau,h,h^{\prime}}\times e_{p^{-1}W(\kappa_r)^{\times}})\tr(\gamma_0|\xi)\ .
\end{aligned}\]
\end{thm}

Note that Lemma \ref{LocalComparison} shows that the functions $\phi_{\tau,h,h^{\prime}}\times e_{p^{-1}W(\kappa_r)^{\times}}$ and $f_{\tau,h}^{\vee}\times h^{\prime}\times e_{p^{-r^{\prime}}\mathbb{Z}_p^{\times}}$ are associated.

Using this expression and going into the calculations of \cite{KottwitzLambdaAdic} that exploit the process of pseudostabilization applicable to the Shimura varieties considered, we can rewrite the last theorem in the following form.

\begin{cor}\label{TraceCohomology} We have the following equality:
\[
n \tr(\tau\times h^{\vee}\times h^{\prime}\times e_{\mathbb{Z}_p^{\times}}\times f^p|[H_{\xi}]) = \tr(f_{\tau,h}^{\vee}\times h^{\prime}\times e_{p^{-r^{\prime}}\mathbb{Z}_p^{\times}}\times f^p|[H_{\xi}])\ .
\]
\end{cor}

\begin{proof} We just go through the most important steps in the calculation. Recall that Kottwitz constructs in \cite{KottwitzLambdaAdic} a function $f_{\infty}$ on $\mathbf{G}(\mathbb{R})$, depending on $\xi$. Let
\[
f=f_{\tau,h}^{\vee}\times h^{\prime}\times e_{p^{-r^{\prime}}\mathbb{Z}_p^{\times}}\times f^p\times f_{\infty}
\]
be the function on $\mathbf{G}(\mathbb{A})$. Then arguing as on pages 661-663 of \cite{KottwitzLambdaAdic}, we see that the left hand-side equals
\[
n \tau(\mathbf{G}) \sum_{\gamma_0} SO_{\gamma_0}(f)\ ,
\]
where $\tau(\mathbf{G})$ is the Tamagawa number of $\mathbf{G}$, $\gamma_0$ runs through the stable conjugacy classes in $\mathbf{G}(\mathbb{Q})$ and $SO_{\gamma_0}(f)$ is a stable orbital integral. Now the Arthur-Selberg trace formula shows that this equals
\[
n \sum_{\pi} m(\pi) \tr(f|\pi)\ ,
\]
where $\pi$ runs through automorphic representations for $\mathbf{G}$ with correct central character. Using Lemma 4.2 of \cite{KottwitzLambdaAdic}, this can be rewritten as
\[
\sum_{\pi_f} \tr(f_{\tau,h}^{\vee}\times h^{\prime}\times e_{p^{-r^{\prime}}\mathbb{Z}_p^{\times}}\times f^p|\pi_f) \sum_{\pi_{\infty}} m(\pi_f\otimes \pi_{\infty}) ep(\pi_{\infty}\otimes \xi)\ ,
\]
using the notation $ep$ to denote the Euler-Poincare characteristic as in \cite{KottwitzLambdaAdic}. Here we have used that the $N$ occuring in Lemma 4.2 of \cite{KottwitzLambdaAdic} equals $n$ in our situation. Finally, Matshushima's formula reveals that this equals
\[
\tr(f_{\tau,h}^{\vee}\times h^{\prime}\times e_{p^{-r^{\prime}}\mathbb{Z}_p^{\times}}\times f^p|[H_{\xi}])\ .
\]
\end{proof}

\section{Galois representations attached to automorphic forms}

In this section, we will associate $\ell$-adic Galois representations to (most) regular algebraic conjugate self-dual cuspidal automorphic representations of $\GL_n/\mathbb{F}$, and prove a local-global compatibility result for them. The construction is based on Clozel's base change from $\mathbf{G}$ to $\GL_n$, which we will use in the form given in Section VI of \cite{HarrisTaylor}.

For any irreducible admissible representation $\pi_f$ of $\mathbf{G}(\mathbb{A}_f)$, we define the isotypic component
\[
W_{\xi}^{\ast}(\pi_f)=\mathrm{Hom}_{\mathbf{G}(\mathbb{A}_f)}(\pi_f,H_{\xi}^{\ast})
\]
of $H_{\xi}^{\ast}$. Taking the alternating sum gives an element $[W_{\xi}(\pi_f)]$ in the Grothendieck group of $\ell$-adic representations of $\mathrm{Gal}(\overline{\mathbb{F}}/\mathbb{F})$. Kottwitz proves in \cite{KottwitzLambdaAdic} that $\pi_f$ occurs either only in even or only in odd degrees, so that in particular no cancellation can occur and $\pm [W_{\xi}(\pi_f)]$ is effective. Also recall the integer $a(\pi_f)$ from \cite{KottwitzLambdaAdic}, which is roughly the multiplicity of $\pi_f$: In fact, we have
\[
\mathrm{dim}\ [W_{\xi}(\pi_f)] = a(\pi_f)n\ .
\]

\begin{cor}\label{LocalGlobalComp} Assume that $\pi_f$ is an irreducible admissible representation of $\mathbf{G}(\mathbb{A}_f)$ with $W_{\xi}^{\ast}(\pi_f)\neq 0$. Let $\pi_p = \pi_w\otimes \pi_p^w\otimes \pi_{p,0}$ corresponding to
\[
\mathbf{G}(\mathbb{Q}_p)=\GL_n(F)\times D^{\prime\times} \times \mathbb{Q}_p^{\times}\ ,
\]
assume that $\pi_{p,0}$ is unramified and let $\chi_{\pi_{p,0}}$ be the character of $W_F\subset W_{\mathbb{Q}_p}$ corresponding to $\pi_{p,0}$ by local class-field theory. Then for all $\tau\in W_F$ projecting to a positive power of Frobenius and $h\in C_c^{\infty}(\GL_n(\mathcal{O}))$, we have
\[
\tr(f_{\tau,h}^{\vee}|\pi_w) = \frac 1{a(\pi_f)}\tr(\tau|[W_{\xi}(\pi_f)]\otimes \chi_{\pi_{p,0}}) \tr(h^{\vee}|\pi_w)\ .
\]
\end{cor}

\begin{proof} Take some function $h^{\prime}\in C_c^{\infty}(\mathcal{O}_{D^{\prime}}^{\times})$ such that $\tr(h^{\prime}|\pi_p^w)=1$, and take $m\geq 1$ such that both $h^{\vee}\times h^{\prime}\times e_{\mathbb{Z}_p^{\times}}$ and $f_{\tau,h}^{\vee}\times h^{\prime}\times e_{\mathbb{Z}_p^{\times}}$ are bi-$K_p^m$-invariant.

Further, take a compact open subgroup $K^p\subset \mathbf{G}(\mathbb{A}_f^p)$ such that $\pi_f^p$ has $K^p$-invariants. Because $H^{\ast}(\mathrm{Sh}_K,\mathcal{F}_K)$ is finite-dimensional, there are only finitely many irreducible admissible representations $\pi_f^{\prime}$ with invariants under $K=K_p^mK^p$ and $W_{\xi}^{\ast}(\pi_f^{\prime})\neq 0$.

Now Corollary VI.2.3 of \cite{HarrisTaylor} implies that there exists a function $f^p\in C_c^{\infty}(\mathbf{G}(\mathbb{A}_f^p))$ biinvariant under $K^p$ with $\tr(f^p|\pi_f^p)=1$ and such that whenever $\pi_f^{\prime}$ is an irreducible admissible representation of $\mathbf{G}(\mathbb{A}_f)$ with $W_{\xi}^{\ast}(\pi_f^{\prime})\neq 0$, with invariants under $K_p^mK^p$, and $\tr(f^p|\pi_f^{\prime p})\neq 0$, then $\pi_f^{\prime}\cong \pi_f$. Indeed, it is always possible to find $f^p$ such that the conclusion $\pi_f^{\prime p}\cong \pi_f^p$ holds true, but Corollary VI.2.3 of \cite{HarrisTaylor} tells us that then already $\pi_f^{\prime}\cong \pi_f$. Fix such an $f^p$.

We apply Corollary \ref{TraceCohomology} to these functions. Both sides of the equation reduce to the contribution of $W_{\xi}^{\ast}(\pi_f)\otimes \pi_f$. Direct inspection reveals that the left-hand side is
\[
n \tr(\tau|[W_{\xi}(\pi_f)])\tr(h^{\vee}|\pi_w)\ ,
\]
and the right-hand side gives
\[
n a(\pi_f) \tr(f_{\tau,h}^{\vee}|\pi_w)\pi_{p,0}(p^{-r^{\prime}})\ .
\]
Rewriting gives the corollary.
\end{proof}

It is useful to be able to embed local components $\pi_w$ into a global representation $\pi_f$. This is achieved by the following theorem.

\begin{thm}\label{EmbeddingLocalGlobalRepr} Assume that $\pi$ is an irreducible smooth representation of $\GL_n(F)$ that is either essentially square-integrable or a generalized Speh representation. Then there is an irreducible admissible representation $\pi_f$ of $\mathbf{G}(\mathbb{A}_f)$ and an algebraic representation $\xi$ of $\mathbf{G}$ such that $W_{\xi}^{\ast}(\pi_f)\neq 0$, the component $\pi_{p,0}$ is unramified and such that the component $\pi_w$ of $\pi_f$ is an unramified twist of $\pi$.
\end{thm}

\begin{proof} This follows from Corollary VI.2.5 and Lemma VI.2.11 of \cite{HarrisTaylor}.
\end{proof}

\begin{cor} If $\pi$ is an irreducible smooth representation of $\GL_n(F)$ that is either essentially square-integrable or a generalized Speh representation, then there is a $\mathbb{Q}$-linear combination $\mathrm{rec}(\pi)$ of representations of $W_F$ with nonnegative coefficients such that for all $\tau$, $h$, we have
\[
\tr(f_{\tau,h}|\pi) = \tr(\tau|\mathrm{rec}(\pi))\tr(h|\pi)\ .
\]
\end{cor}

\begin{rem} As we can always write a $p$-adic field as $\mathbb{F}_w$ with all the assumptions on $\mathbb{F}$ etc., this checks condition (ii) in Lemma \ref{LemmaToProve}.
\end{rem}

\begin{proof} Apply Theorem \ref{EmbeddingLocalGlobalRepr} to $\pi^{\vee}$ and Corollary \ref{LocalGlobalComp} to the resulting representations $\pi_f$, $\xi$. Then $\pi_w = \pi^{\vee}\otimes \chi$ for some unramified character $\chi$. The representation
\[
\mathrm{rec}(\pi) = \frac 1{a(\pi_f)} [W_{\xi}(\pi_f)]^{\vee}\otimes \chi_{\pi_{p,0}}^{-1}\otimes \chi^{-1}
\]
satisfies the requirement.
\end{proof}

Furthermore, we have the following theorem from \cite{HarrisTaylor} constructing irreducible admissible representations $\pi_f$ of $\mathbf{G}(\mathbb{A}_f)$ with $W_{\xi}^{\ast}(\pi_f)\neq 0$ from cuspidal automorphic representations of $\GL_n/\mathbb{F}$.

\begin{thm}\label{GLnToG} Let $\Pi$ be a cuspidal automorphic representation of $\GL_n/\mathbb{F}$ such that
\begin{altenumerate}
\item[{\rm (i)}] $\Pi^{\vee} = \Pi\circ c$, where $c: \GL_n(\mathbb{A}_{\mathbb{F}})\longrightarrow \GL_n(\mathbb{A}_{\mathbb{F}})$ is complex conjugation;
\item[{\rm (ii)}] $\Pi_{\infty}$ is regular algebraic, i.e. it has the same infinitesimal character as an algebraic representation of $\mathrm{Res}_{\mathbb{F}/\mathbb{Q}}(\GL_n)$ over $\mathbb{C}$;
\item[{\rm (iii)}] $\Pi_x$ is square-integrable.
\end{altenumerate}
Then there exists an irreducible admissible representation $\pi_f$ of $\mathbf{G}(\mathbb{A}_f)$, an algebraic representation $\xi$ of $\mathbf{G}$ such that $W_{\xi}^{\ast}(\pi_f)\neq 0$, and a character $\psi$ of $\mathbb{K}^{\times}\backslash \mathbb{A}_{\mathbb{K}}^{\times}$ whose infinite components are algebraic such that $\pi_{w^{\prime}}\cong \Pi_{w^{\prime}}$ and $\pi_{p^{\prime},0} = \psi_{u^{\prime}}$ for all places $w^{\prime}$, $u^{\prime}$ and $p^{\prime}$ of $\mathbb{F}$, resp. $\mathbb{K}$, resp. $\mathbb{Q}$, such that $w^{\prime}|u^{\prime}|p^{\prime}$, $p^{\prime}$ splits in $\mathbb{K}$ and $w^{\prime}\neq x,x^c$.

Moreover, for any given prime $p^{\prime}$ as above, one can arrange that $\pi_{p^{\prime},0}$ is unramified.
\end{thm}

\begin{proof} All references refer to \cite{HarrisTaylor}. Use Theorem VI.1.1 to produce an automorphic representation of $\mathbb{D}^{\times}$. It continues to have properties analogous to (i) and (ii). Then Lemma VI.2.10, Theorem VI.2.9 and the properties of the base-change map established in Theorem VI.2.1 finish the proof.
\end{proof}

Combining this with Corollary \ref{LocalGlobalComp}, we get the following construction of Galois representations. In its formulation, we assume that parts (a) and (b) of Theorem \ref{MainTheorem2} are true; we will only use this corollary after we have proven these assertions. We prefer to state it now as it fits in with the discussion here.

\begin{thm}\label{ConstructionGaloisRepr} Let $\Pi$ be a cuspidal automorphic representation of $\GL_n/\mathbb{F}$ such that
\begin{altenumerate}
\item[{\rm (i)}] $\Pi^{\vee} = \Pi\circ c$;
\item[{\rm (ii)}] $\Pi_{\infty}$ is regular algebraic;
\item[{\rm (iii)}] $\Pi_x$ is square-integrable.
\end{altenumerate}
Then there exists an integer $a\geq 1$ and an $\ell$-adic representation $R(\Pi)$ of $\mathrm{Gal}(\bar{\mathbb{F}}/\mathbb{F})$ of dimension $an$ such that for all finite places $v$ of $\mathbb{F}$ whose residue characteristic is different from $\ell$, we have
\[
R(\Pi)|_{W_{\mathbb{F}_v}} = a\cdot \mathrm{rec}(\Pi_v)
\]
as elements of the Grothendieck group of representations of $W_{\mathbb{F}_v}$.
\end{thm}

\begin{proof} Choose $\pi_f$, $\xi$ and $\psi$ as in the theorem, and take $a=a(\pi_f)$ as well as
\[
R(\Pi)= [W_{\xi}(\pi_f)]^{\vee}\otimes \chi_{\psi}^{-1}\ ,
\]
where $\chi_{\psi}$ corresponds to $\psi$ by global class-field theory. Because $\psi$ is algebraic at all infinite places, this makes sense.

The desired property is now immediate for all $v\neq x,x^c$ which are split over $\mathbb{F}_0$ and for which $\pi_{p^{\prime},0}$ is unramified, where $p^{\prime}$ is the rational prime below $v$. By the Chebotarev density theorem, this determines the virtual representation $\frac 1a R(\Pi)$ uniquely. Now, since for any given $p^{\prime}$, one can choose $\pi_f$, $\xi$ and $\psi$ so that $\pi_{p^{\prime},0}$ is unramified, the condition that $\pi_{p^{\prime},0}$ be unramified can be dropped.

Now if $v$ is arbitrary, take some real quadratic field $R$ linearly disjoint from $\mathbb{F}$ such that the rational prime below $x$ splits in $R$, and the completions of $R$ and $\mathbb{K}$ at the rational prime below $v$ are isomorphic. Let $\tilde{\mathbb{F}}_0=\mathbb{F}_0R$ and $\tilde{\mathbb{F}}=\mathbb{F}R$, and choose places $\tilde{v}$ and $\tilde{x}$ above $v$ and $x$ in $\tilde{\mathbb{F}}$; we can assume that $\tilde{v}\neq \tilde{x},\tilde{x}^c$. Moreover, the assumptions imply that $\tilde{v}$ is split over $\tilde{\mathbb{F}}_0$. Also $\tilde{\mathbb{F}}_{\tilde{v}} = \mathbb{F}_v$.

Taking the base-change $\tilde{\Pi}$ of $\Pi$ to $\tilde{\mathbb{F}}$, it continues to have the required properties, hence we get a representation $R(\tilde{\Pi})$ and an integer $\tilde{a}\geq 1$. From Chebotarev density, it follows that $\frac 1{\tilde{a}}R(\tilde{\Pi}) = \frac 1a R(\Pi)|_{\mathrm{Gal}(\mathbb{F}/\tilde{\mathbb{F}})}$. Now applying the previous result at the place $\tilde{v}$ gives the desired result for $\Pi$ at $v$.
\end{proof}

Later we need a slight variant of Theorem \ref{EmbeddingLocalGlobalRepr} for $\GL_n/\mathbb{F}$.

\begin{thm}\label{EmbLocalGlobalReprGLn} Let $\pi$ be an essentially square-integrable irreducible smooth representation of $\GL_n(F)$. Then there exists a cuspidal automorphic representation $\Pi$ of $\GL_n/\mathbb{F}$ such that
\begin{altenumerate}
\item[{\rm (i)}] $\Pi^{\vee} = \Pi\circ c$;
\item[{\rm (ii)}] $\Pi_{\infty}$ is regular algebraic;
\item[{\rm (iii)}] $\Pi_x$ is supercuspidal;
\end{altenumerate}
and such that the component $\Pi_w$ is an unramified twist of $\pi$.
\end{thm}

\begin{proof} This follows from Corollary VI.2.6 of \cite{HarrisTaylor}.
\end{proof}

\section{Relation to the Lubin-Tate tower}

In this section, we are back in the local situation. Let
\[
\beta\in \GL_n(\mathcal{O}_r)\mathrm{diag}(\varpi,1,\ldots,1)\GL_n(\mathcal{O}_r)
\]
be basic. Then the tower
\[
(\mathrm{Spf}\ R_m)_{m\geq 0}=(\mathrm{Spf}\ R_{\beta,m}\otimes_{\mathcal{O}_r} \breve{\mathcal{O}})_{m\geq 0}
\]
is called the Lubin-Tate tower. It does not depend on $\beta$, as the base-change of $\overline{H}_{\beta}$ to $\bar{\kappa}$ does not depend on $\beta$: it is the unique one-dimensional formal $\mathcal{O}$-module $\tilde{H}$ of height $n$ over $\bar{\kappa}$.

We describe compatible actions of
\[
\GL_n(\mathcal{O})\times (\mathcal{D}^{\times}\times W_F)_0
\]
on $\mathrm{Spf}\ R_m$, where
\[
(\mathcal{D}^{\times}\times W_F)_0 = \{(d,\tau)\in \mathcal{D}^{\times}\times W_F\mid v(d) + v(\tau)=0\}\ .
\]
It is clear how $\GL_n(\mathcal{O})$ acts. Now, for $(d,\tau)\in (\mathcal{D}^{\times}\times W_F)_0$ with $\tau\in \mathrm{Frob}^r I_F$, one gets quasi-isogenies $d^{-1}: \tilde{H}\longrightarrow \tilde{H}$ of height $r$, $F: \tilde{H}\longrightarrow F^{\ast}\tilde{H}$ of height $1$ and $\tau: \mathrm{Frob}^{r \ast}\tilde{H}\longrightarrow \tilde{H}$ of height $0$. Hence the composition $d^{-1}\times F^{-r}\times \tau:\tilde{H}\longrightarrow \tilde{H}$ defines a quasi-isogeny of degree $0$, i.e. an automorphism. We get a corresponding action on $\mathrm{Spf}\ R_m$.

Let
\[
R\psi = \lim\limits_{\longrightarrow} H^0(R\psi_{\mathrm{Spf}\ R_m}\bar{\mathbb{Q}}_{\ell})
\]
be the global sections of the nearby cycle sheaves associated to the Lubin-Tate tower. They carry an action of $\GL_n(\mathcal{O})\times (\mathcal{D}^{\times}\times W_F)_0$. It is known that this extends even further, using Hecke operators, but we will not need this part of the action.

The following theorem is well-known.

\begin{thm}\label{LubinTateCohoNice} Each representation $R^i\psi$ is an admissible smooth representation of $\GL_n(\mathcal{O})$, a smooth representation of $\mathcal{O}_\mathcal{D}^{\times}$, and a continuous representation of $I_F$. Moreover, $R^i\psi$ vanishes outside the range $0\leq i\leq n-1$.
\end{thm}

\begin{proof} Everything except the smoothness of the $\mathcal{O}_\mathcal{D}^{\times}$-action follows from Theorem \ref{RepresentationNice}. This smoothness follows from a result of Berkovich, Corollary 4.5 of \cite{Berkovich2}, cf. proof of \cite{HarrisTaylor}, Lemma II.2.8.
\end{proof}

Note that since $\beta$ is basic, we can use the constructions of Section \ref{NormMaps} to get an associated norm $d=N\beta\in \mathcal{D}^{\times}$, well-defined up to $\mathcal{O}_\mathcal{D}^{\times}$-conjugation. Then $(d^{-1},\tau)\in (\mathcal{D}^{\times}\times W_F)_0$ if and only if $\tau\in \mathrm{Frob}^r I_F$. It is now an easy exercise to verify the following proposition.

\begin{prop}\label{CompLubinTate} For all $h\in C_c^{\infty}(\GL_n(\mathcal{O}))$, the equality
\[
\phi_{\tau,h}(\beta) = \tr(\tau\times d^{-1}\times h^{\vee}|[R\psi])
\]
holds.
\end{prop}

Now consider
\[
\tilde{[R\psi]} = \text{c-Ind}_{\GL_n(\mathcal{O})\times (\mathcal{D}^{\times}\times W_F)_0}^{\GL_n(\mathcal{O})\times \mathcal{D}^{\times}\times W_F} [R\psi]\ .
\]
In particular, for any irreducible representation $\rho$ of $\mathcal{D}^{\times}$, the space
\[
[R\psi](\rho) = \mathrm{Hom}_{\mathcal{D}^{\times}} (\rho,\tilde{[R\psi]}) = \mathrm{Hom}_{\mathcal{O}_\mathcal{D}^{\times}}(\rho|_{\mathcal{O}_\mathcal{D}^{\times}},[R\psi])
\]
carries an action of $\GL_n(\mathcal{O})\times W_F$. Let $\pi=\mathrm{JL}(\rho)$ be the associated representation of $\GL_n(F)$ via the Jacquet-Langlands correspondence, characterized by the equality
\[
\tr(g|\pi) = (-1)^{n-1} \tr(d|\rho)
\]
for all regular elliptic $g\in \GL_n(F)$ and $d\in \mathcal{D}^{\times}$ such that $g$ and $d$ have the same characteristic polynomial.

Recall that we have already shown that there is a nonnegative $\mathbb{Q}$-linear combination of $W_F$-representations $\mathrm{rec}(\pi)$ such that for all $\tau$ and $h$
\[
\tr(f_{\tau,h}|\pi) = \tr(\tau|\mathrm{rec}(\pi)) \tr(h|\pi)\ .
\]
The following corollary translates this property into a statement about the cohomology of the Lubin-Tate tower.

\begin{cor}\label{LubinTateCoho} Let $\rho$ be an irreducible representation of $\mathcal{D}^{\times}$ such that $\pi=\mathrm{JL}(\rho)$ is supercuspidal. Then, as a virtual representation of the group $\GL_n(\mathcal{O})\times W_F$, the representation $[R\psi](\rho^{\vee})$ is equal to $(-1)^{n-1} \pi^{\vee}|_{\GL_n(\mathcal{O})}\otimes \mathrm{rec}(\pi)$.
\end{cor}

\begin{rem} One can define a natural action of $\GL_n(F)\times W_F$ on $[R\psi](\rho^{\vee})$, and the proposition stays true when considering the action of $\GL_n(F)\times W_F$ on both objects. However, our methods do not prove this.
\end{rem}

\begin{proof} Let $h\in C_c^{\infty}(\GL_n(\mathcal{O}))$ and $\tau\in \mathrm{Frob}^r I_F$. Let $\Pi$ be the base-change of $\pi$ to $\GL_n(F_r)$. Note that if
\[
\beta\in \GL_n(\mathcal{O}_r)\mathrm{diag}(\varpi,1,\ldots,1)\GL_n(\mathcal{O}_r)
\]
lies in the support of the twisted character $\Theta_{\Pi,\sigma}$ of $\Pi$, then $\beta$ is basic, i.e. $\beta\in B_r$. We compute, using Corollary \ref{IntegrationLemma},
\[\begin{aligned}
\tr(\tau|\mathrm{rec}(\pi)) \tr(h|\pi) = \tr(f_{\tau,h}|\pi) &= \tr((\phi_{\tau,h},\sigma)|\Pi)\\
&=\int_{B_r} \phi_{\tau,h}(\beta) \Theta_{\Pi,\sigma}(\beta) d\beta\\
&=(-1)^{n-1} \int_{B_r} \tr(\tau\times N\beta^{-1}\times h^{\vee}|[R\psi]) \Theta_{\rho}(N\beta) d\beta\\
&=(-1)^{n-1} \int_{\mathcal{D}_r} \tr(\tau\times d^{-1}\times h^{\vee}|[R\psi]) \Theta_{\rho}(d) dd\\
&=(-1)^{n-1} \tr(\tau\times h^{\vee}|[R\psi](\rho))\ .
\end{aligned}\]
This exactly proves the corollary.
\end{proof}

Now we can check condition (iii) of Lemma \ref{LemmaToProve}, finishing the proof of parts (a) and (b) of Theorem \ref{MainTheorem2}.

\begin{cor} For any supercuspidal representation $\pi$ of $\GL_n(F)$, the virtual representation $\mathrm{rec}(\pi)$ can be represented by a $\mathbb{Z}$-linear combination of representations.
\end{cor}

\begin{proof} The theory of newvectors for $\GL_n$, \cite{JacquetPiatetskiShapiroShalika}, exhibits a representation of $\GL_n(\mathcal{O})$ that occurs with multiplicity $1$ in $\pi^{\vee}|_{\GL_n(\mathcal{O})}$. By the last proposition, the corresponding isotypic component of $[R\psi](\mathrm{JL}^{-1}(\pi)^{\vee})$ is a $\mathbb{Z}$-linear combination of representations of $W_F$ that represents $(-1)^{n-1}\mathrm{rec}(\pi)$.
\end{proof}

\section{Bijectivity of the correspondence}\label{Bijectivity}

For an irreducible smooth representation $\pi$ of $\GL_n(F)$, let $\sigma(\pi) = \mathrm{rec}(\pi)(\frac{1-n}2)$ be the associated Weil group representation, normalized appropriately. In this section, we will prove that $\pi\mapsto \sigma(\pi)$ gives a bijection between irreducible supercuspidal representations of $\GL_n(F)$ and irreducible representations of $W_F$ of dimension $n$.

\begin{thm}\label{Functoriality} The correspondence $\pi\mapsto \sigma(\pi)$ has the following functorial properties.
\begin{altenumerate}
\item[{\rm (i)}] If $n=1$, then $\sigma(\pi)$ is given by local class-field theory.

\item[{\rm (ii)}] If $\pi$ is an irreducible subquotient of the normalized parabolic induction of the irreducible representation $\pi_1\otimes\cdots\otimes\pi_t$ of $\GL_{n_1}(F)\times\cdots\times\GL_{n_t}(F)$, then
\[
\sigma(\pi)=\sigma(\pi_1)\oplus\ldots\oplus\sigma(\pi_t)\ .
\]

\item[{\rm (iii)}] If $\chi$ is a character of $F^{\times}$, then $\sigma\left(\pi\otimes \chi\circ \det\right) = \sigma(\pi)\otimes \sigma(\chi)$.

\item[{\rm (iv)}] If $F^{\prime}/F$ is a cyclic Galois extension of prime degree, if $\pi$ is an irreducible supercuspidal representation of $\GL_n(F)$, and $\Pi$ is the base-change lift to $\GL_n(F^{\prime})$, then
\[
\sigma(\Pi) = \sigma(\pi)|_{W_{F^{\prime}}}\ .
\]

\item[{\rm (v)}] If $\sigma(\pi)$ is unramified, then $\pi$ has an Iwahori-fixed vector (equivalently, the supercuspidal support of $\pi$ consists of unramified characters).
\end{altenumerate}
\end{thm}

\begin{proof} Part (i) follows directly from the comparison with Lubin-Tate tower, and the well-known description of its cohomology. Alternatively, one can deduce it from global class-field theory and the local-global compatibility result for $\mathrm{rec}(\pi)$.

Part (ii) is part (b) of Theorem \ref{MainTheorem2}, which we have already proved.

Assertions (iii) and (iv) are Lemma VII.2.1 and Lemma VII.2.4 of \cite{HarrisTaylor}, respectively. Their proofs carry over without change, using our Theorem \ref{ConstructionGaloisRepr} instead of Theorem VII.1.9 of \cite{HarrisTaylor}.

Finally, part (v) is a corollary of our previous work. Using parts (i) and (ii), one immediately reduces to the case that $\pi$ is a supercuspidal representation of $\GL_n(F)$ for $n\geq 2$. In that case, we prove that $\sigma(\pi)^{I_F}=0$.

For this purpose, introduce for any $r\geq 1$ the function $\phi_{r,h} = \int_{\mathrm{Frob}^r I_F} \phi_{\tau,h} d\tau$; this makes sense as the $W_{F_r}$-representations involved are continuous. In fact, one easily checks that
\[
\phi_{r,h}(\beta) = \mathrm{tr}^{\mathrm{ss}}(\mathrm{Frob}^r\times h^{\vee} | [R\psi_{\beta}])\ ,
\]
where we use the semisimple trace as in \cite{ScholzeGLn}.

We claim that $\phi_{r,h}(\beta)$ is constant for all basic
\[
\beta\in \GL_n(\mathcal{O}_r)\mathrm{diag}(\varpi,1,\ldots,1)\GL_n(\mathcal{O}_r)\ .
\]
Let $R_{I_F}R\psi_{\beta}$ be the derived invariants of $R\psi_{\beta}$ under $I_F$. The cohomology groups of this complex carry an admissible action of $\mathrm{Frob}^{r\mathbb{Z}}\times \GL_n(\mathcal{O})$. Then Lemma 7.5 of \cite{Scholze} says that
\[
(1-q^r)\mathrm{tr}^{\mathrm{ss}}(\mathrm{Frob}^r\times h^{\vee} | [R\psi_{\beta}]) = \mathrm{tr}(\mathrm{Frob}^r\times h^{\vee} | [R_{I_F}R\psi_{\beta}])\ ,
\]
where $q$ is the cardinality of $\kappa$. Therefore it suffices to show that the cohomology groups of the complex $R_{I_F}R\psi_{\beta}$ are independent of $\beta$. But this follows from the description of these groups given in Corollary 5.6 of \cite{ScholzeGLn}.\footnote{At least in the case $F=\mathbb{Q}_p$, but the description generalizes without problems. Note that one may use the algebraizations constructed in Theorem \ref{Algebraization} instead of the Shimura varieties considered in \cite{ScholzeGLn}.}

Now if $\pi$ is a supercuspidal representation of $\GL_n(F)$ with base-change lift $\Pi$ to $\GL_n(F_r)$ and associated representation $\rho$ of $\mathcal{D}^{\times}$ via the Jacquet-Langlands correspondence, then as in the proof of Corollary \ref{LubinTateCoho},
\[\begin{aligned}
\tr(\mathrm{Frob}^r|\sigma(\pi)^{I_F})\tr(h|\pi) &= \int_{\mathrm{Frob}^r I_F}\tr(\tau|\sigma(\pi))\tr(h|\pi) d\tau = \int_{\mathrm{Frob}^r I_F}\tr(f_{\tau,h}|\pi) d\tau\\
&= \int_{\mathrm{Frob}^r I_F} \tr((\phi_{\tau,h},\sigma)|\Pi) d\tau = \tr((\phi_{r,h},\sigma)|\Pi)\\
&= \int_{B_r} \phi_{r,h}(\beta) \Theta_{\Pi,\sigma}(\beta) d\beta = \mathrm{const.} \int_{B_r} \Theta_{\rho}(N\beta) d\beta\\
&= \mathrm{const.} \int_{\mathcal{D}_r} \Theta_{\rho}(d) dd = 0\ ,
\end{aligned}\]
because $\rho$ has no invariants under $\mathcal{O}_\mathcal{D}^{\times}$.
\end{proof}

We axiomatize the situation.

\begin{definition} We say that a family of maps associating to any finite extension $F/\mathbb{Q}_p$ and any irreducible smooth representation $\pi$ of $\GL_n(F)$ an $n$-dimensional representation $\sigma(\pi)$ of $W_F$ is a functorial extension of class-field theory if it is natural in the sense that it commutes with isomorphisms $F\cong F^{\prime}$ and has the properties (i) through (v) of the last theorem.
\end{definition}

Our results show that the map $\pi\longmapsto \sigma(\pi)$ is a functorial extension of class-field theory.

\begin{thm}\label{FunctorialBijective} For any functorial extension of local class-field theory, the map from irreducible supercuspidal representations of $\GL_n(F)$ to $n$-dimensional representations of $W_F$ is injective with image consisting exactly of the irreducible $n$-dimensional representations of $W_F$.
\end{thm}

\begin{proof} First, it is clear that if there exists a functorial extension of local class-field theory, then any irreducible smooth representation of $\GL_n(F)$ has an Iwahori-fixed vector after a finite series of cyclic base-changes. In particular, supercuspidal representations become unramified after a finite series of cyclic base-changes, as they always stay (up to twist) unitarily induced from supercuspidals.

Now we prove the theorem by induction on $n$, the case $n=1$ being obvious.

First, we check that if $\pi$ is supercuspidal, then $\sigma(\pi)$ is irreducible. Choose a series of cyclic extensions of prime degree $F=F_0\subset F_1\subset \ldots \subset F_m$ such that the base-change of $\pi$ from $F$ to $F_m$ is not supercuspidal, but the base-change to $F_{m-1}$ is still supercuspidal. Replacing $F$ by $F_{m-1}$, we can assume that there is a cyclic extension of prime degree $F_1/F$ such that the base-change $\Pi$ of $\pi$ from $F$ to $F_1$ is not supercuspidal.

Let $\tau\in \mathrm{Gal}(F_1/F)\cong \mathbb{Z}/g\mathbb{Z}$ be a generator. The results of Arthur-Clozel, \cite{ArthurClozel}, Lemma 6.10, imply that $\Pi = \Pi_1\boxplus \Pi^{\tau}\boxplus \ldots \boxplus \Pi_1^{\tau^{g-1}}$ for some supercuspidal representation $\Pi_1$ with $\Pi_1\not\cong \Pi_1^{\tau}$.

By induction on $n$, we know that $\sigma(\Pi_1)$ is irreducible. Further $\sigma(\pi)|_{W_{F_1}}=\sigma(\Pi_1)\oplus\ldots\oplus\sigma(\Pi_1)^{\tau^{g-1}}$, with $\sigma(\Pi_1)\not\cong \sigma(\Pi_1)^{\tau}$, by induction, because $\Pi_1\not\cong \Pi_1^{\tau}$. This implies that $\sigma(\pi)$ is irreducible.

Now we check that the map is bijective. Choose any irreducible $n$-dimensional representation $\sigma$ of $W_F$. We choose a series of cyclic extensions of prime degree $F=F_0\subset F_1\subset \ldots \subset F_m$ such that the restriction of $\sigma$ to $W_{F_m}$ is not irreducible, but the restriction to $W_{F_{m-1}}$ is still irreducible.

In this situation, $\sigma|_{W_{F_m}} = \Sigma\oplus\ldots\oplus\Sigma^{\tau^{g-1}}$, where $\tau$ generates $\mathrm{Gal}(F_m/F_{m-1})\cong \mathbb{Z}/g\mathbb{Z}$, and $\Sigma$ is an irreducible representation of $W_{F_m}$ with $\Sigma\not\cong \Sigma^{\tau}$. By induction, there is a unique supercuspidal representation $\Pi$ of $\GL_{\frac ng}(F_m)$ with $\sigma(\Pi)=\Sigma$. Then $\Pi_m = \Pi\boxplus\ldots\boxplus\Pi^{\tau^{g-1}}$ lifts to a unique representation $\Pi_{m-1}$ of $\GL_n(F_{m-1})$, and this representation satisfies
\[
\sigma(\Pi_{m-1})|_{W_{F_m}} = \Sigma\oplus\ldots\oplus\Sigma^{\tau^{g-1}}\ ,
\]
whence $\sigma(\Pi_{m-1}) = \sigma|_{W_{F_{m-1}}}$: Because $\Sigma\not\cong \Sigma^{\tau}$, there is a unique representation of $W_{F_{m-1}}$ with restriction to $W_{F_m}$ being equal to $\Sigma\oplus\ldots\oplus\Sigma^{\tau^{g-1}}$, namely $\mathrm{Ind}_{W_{F_m}}^{W_{F_{m-1}}} \Sigma$. This discussion shows that there is a supercuspidal representation $\Pi_{m-1}$ of $\GL_n(F_{m-1})$ with $\sigma(\Pi_{m-1})=\sigma|_{W_{F_{m-1}}}$.

We claim that this representation $\Pi_{m-1}$ is unique. Indeed, the base-change $\Pi_m$ of $\Pi_{m-1}$ to $\GL_n(F_m)$ cannot be supercuspidal, as $\sigma(\Pi_m)=\sigma(\Pi_{m-1})|_{W_{F_m}}$ is not irreducible. Hence $\Pi_m=\Pi\boxplus\ldots\boxplus\Pi^{\tau^{g-1}}$ for some supercuspidal representation $\Pi$ of $\GL_{\frac ng}(F_m)$. Then
\[
\sigma(\Pi)\oplus\ldots\oplus\sigma(\Pi)^{\tau^{g-1}} = \Sigma\oplus\ldots\oplus\Sigma^{\tau^{g-1}}\ ,
\]
which implies that after replacing $\Pi$ by $\Pi^{\tau^i}$ for some $i$, we have $\sigma(\Pi)=\Sigma$. By induction on $n$, this determines $\Pi$ uniquely, hence $\Pi_m$ and finally $\Pi_{m-1}$.

Now we use descending induction on $i=m-1,m-2,\ldots,0$. Changing notation if necessary, we can assume that $i=1$. Hence we have an irreducible $n$-dimensional representation $\sigma$ of $W_F$ whose restriction $\sigma_1$ to $W_{F_1}$ is still irreducible, and we know that there exists a unique supercuspidal representation $\pi_1$ of $\GL_n(F_1)$ with $\sigma(\pi_1)=\sigma_1$. Again, $F_1/F$ is a cyclic extension of prime degree $g$, and we let $\tau$ be a generator of $\mathrm{Gal}(F_1/F)$.

Note that $\sigma_1^{\tau}\cong \sigma_1$; hence $\pi_1^{\tau}\cong \pi_1$ by the uniqueness assertion already established. The results of Arthur-Clozel, \cite{ArthurClozel}, Theorem 6.2, then say that $\pi_1$ lifts to a supercuspidal representation $\pi^{\prime}$ of $\GL_n(F)$. We know that
\[
\sigma(\pi^{\prime})|_{W_{F_1}} = \sigma(\pi_1) = \sigma|_{W_{F_1}}\ .
\]
This implies that $\sigma(\pi^{\prime})=\sigma\otimes \chi$ for some character $\chi$ of $\mathrm{Gal}(F_1/F)$, which we also consider as a character of $F^{\times}$ via local class-field theory. In particular, for $\pi=\pi^{\prime}\otimes \chi^{-1}$, we have $\sigma(\pi)=\sigma$, which shows the existence of $\pi$.

Now assume that $\pi$, $\pi^{\prime}$ are two representations with $\sigma(\pi)=\sigma(\pi^{\prime})=\sigma$. Then their respective base-changes $\pi_1$, $\pi_1^{\prime}$ to $F_1$ satisfy $\sigma(\pi_1)=\sigma(\pi_1^{\prime})=\sigma_1$. As $\sigma_1$ is irreducible, both $\pi_1$ and $\pi_1^{\prime}$ are supercuspidal, so that by induction $\pi_1\cong \pi_1^{\prime}$. Now the uniqueness of lifts, \cite{ArthurClozel}, Proposition 6.7, implies that $\pi\cong \pi^{\prime}\otimes \chi$ for some character $\chi$. But then $\sigma\cong \sigma\otimes \chi$, which implies $\chi=1$, as $\sigma_1$ is irreducible.
\end{proof}

\section{Non-galois automorphic induction}

In this section, we repeat Harris' arguments constructing certain cuspidal automorphic representations associated to some Weil group representations which are induced from a character. The crucial steps are presented in a slightly more general context.

We need the following definitions.

\begin{definition} Let $\mathbb{L}$ be a number field.

\begin{altenumerate}
\item[{\rm (i)}] A cuspidal automorphic representation $\Pi$ of $\GL_n/\mathbb{L}$ is called potentially abelian if there exists an $n$-dimensional representation $\Sigma$ of $W_\mathbb{L}$ such that for all finite places $v$ of $\mathbb{L}$, the representations $\Pi_v$ and $\Sigma|_{W_{\mathbb{L}_v}}$ satisfy
\[
\Sigma|_{W_{\mathbb{L}_v}} = \sigma(\Pi_v)\ ,
\]
and for all infinite places, they are associated via the archimedean Local Langlands Correspondence.
\item[{\rm (ii)}] An $n$-dimensional representation $\Sigma$ of $W_\mathbb{L}$ is called algebraic if for all embeddings $L\hookrightarrow \mathbb{C}$, the corresponding restriction to $W_{\mathbb{C}}\cong \mathbb{C}^{\times}$ is a direct sum of characters of the form $z\longmapsto z^p\overline{z}^q$ for some integers $p$, $q$.
\end{altenumerate}
\end{definition}

The name potentially abelian is motivated by the fact that representations of the Weil group become a direct sum of characters after restriction to a subgroup of finite index.

We need to reformulate Theorem \ref{ConstructionGaloisRepr} with a slightly different normalization. Call $\Pi_{\infty}$ regular $L$-algebraic if $\Pi_{\infty}(\frac{n-1}2)$ is regular algebraic. We note that this notion of regularity is better behaved with respect to functorialities: For example, one checks easily that it is compatible with automorphic induction, cf. the general discussion of Buzzard-Gee, \cite{BuzzardGee}. We still use the notation $\mathbb{F}$ to denote a CM field which is the compositum of a totally real field $\mathbb{F}_0$ of even degree over $\mathbb{Q}$ and an imaginary-quadratic field. Recall that we have fixed a place $x$ of $\mathbb{F}$ which is split over $\mathbb{F}_0$.

\begin{cor}\label{ConstructionGaloisRepr2} Let $\Pi$ be a cuspidal automorphic representation of $\GL_n/\mathbb{F}$ such that
\begin{altenumerate}
\item[{\rm (i)}] $\Pi^{\vee} = \Pi\circ c$;
\item[{\rm (ii)}] $\Pi_{\infty}$ is regular $L$-algebraic;
\item[{\rm (iii)}] $\Pi_x$ is square-integrable.
\end{altenumerate}
Then there exists an integer $a\geq 1$ and an $\ell$-adic representation $R(\Pi)$ of $\mathrm{Gal}(\bar{\mathbb{F}}/\mathbb{F})$ of dimension $an$ such that for all primes $v$ of $\mathbb{F}$ whose residue characteristic is different from $\ell$, we have
\[
R(\Pi)|_{W_{\mathbb{F}_v}} = a\cdot \sigma(\Pi_v)
\]
as elements of the Grothendieck group of representations of $W_{\mathbb{F}_v}$.
\end{cor}

\begin{proof} It suffices to find a character $\phi$ of $\mathbb{F}^{\times}\backslash \mathbb{A}_{\mathbb{F}}^{\times}$ such that $\phi^{-1} = \phi\circ c$ and $\phi_{\infty}(\frac{n-1}2)$ is algebraic. We refer to the easy verification of its existence in \cite{HarrisTaylor}, proof of Corollary VII.2.8.
\end{proof}

\begin{lem} If $\Pi$ is potentially abelian, associated to $\Sigma$, and $\Pi_{\infty}$ is regular $L$-algebraic, then $\Sigma$ is algebraic.
\end{lem}

\begin{proof} This is a direct consequence of the archimedean Local Langlands Correspondence.
\end{proof}

\begin{lem}\label{AlgebraicCharEllAdic} Let $\Sigma$ be an algebraic representation of $W_\mathbb{L}$ and let $\ell$ be a prime number. Fix an isomorphism $\iota: \bar{\mathbb{Q}}_{\ell}\cong \mathbb{C}$. Then there is an $\ell$-adic representation $\sigma_{\ell}$ of $\mathrm{Gal}(\bar{\mathbb{L}}/\mathbb{L})$ such that for all finite places $v$ of $\mathbb{L}$ such that $v$ does not divide $\ell$, the restrictions $\Sigma|_{W_{\mathbb{L}_v}}$ and $\sigma_{\ell}|_{W_{\mathbb{L}_v}}$ are identified via $\iota$.
\end{lem}

\begin{proof} In the case of characters, this can be done by a simple direct construction, made explicit e.g. by Fargues in \cite{FarguesAbelianMotives}, page 4. The general case can be reduced to this special case by passing to a finite extension and descending back to $\mathbb{L}$. Another possibility is to use Langlands' work, \cite{LanglandsTori}, on automorphic forms on tori, as explained in Section 4 of \cite{BuzzardGee}: Note that the representation $\Sigma$ of $W_\mathbb{L}$ necessarily factors through the $L$-group of a torus.
\end{proof}

Now we are ready to state the theorem that will enable us to construct the non-Galois automorphic induction.

\begin{thm}\label{InductionPreparation} Let $\Pi$ be a cuspidal automorphic representation of $\GL_n/\mathbb{F}$ such that $\Pi^{\vee}\cong \Pi\circ c$, $\Pi_{\infty}$ is regular $L$-algebraic, and $\Pi_x$ is supercuspidal.

\begin{altenumerate}
\item[{\rm (i)}] If there exists an $n$-dimensional representation $\Sigma$ of $W_\mathbb{F}$ such that $\Pi$ and $\Sigma$ are associated at all infinite places and for all almost all finite places, then for all finite places $v$, we have
\[
\Sigma|_{W_{\mathbb{F}_v}} = \sigma(\Pi_v)\ ,
\]
i.e., $\Pi$ is potentially abelian.

\item[{\rm (ii)}] Let $\mathbb{F}^{\prime}/\mathbb{F}$ be a cyclic extension of prime degree and assume that there is only one place $x^{\prime}$ of $\mathbb{F}^{\prime}$ above $x$. Further, assume that denoting by $\Pi^{\prime}$ the base-change of $\Pi$ to $\mathbb{F}^{\prime}$, that $\Pi^{\prime}_{x^{\prime}}$ is supercuspidal, and that $\Pi^{\prime}$ is potentially abelian.

Then $\Pi$ is potentially abelian.
\end{altenumerate}
\end{thm}

\begin{proof} For (i), there are systems of $\ell$-adic representations $\sigma_{\ell}^{1}$, $\sigma_{\ell}^{2}$ associated to both $\Sigma$ and $\Pi$, by Lemma \ref{AlgebraicCharEllAdic} and Corollary \ref{ConstructionGaloisRepr2}, respectively. These agree at almost all places by assumption, hence by Chebotarev $\sigma_{\ell}^{1} = \sigma_{\ell}^{2}$. Now one uses that $\Sigma$ and $\sigma_{\ell}^{1}$ are associated for all finite places $v$ which do not divide $\ell$, and the corresponding result for $\Pi$.

For (ii), let $\Pi^{\prime}$ be associated to $\Sigma^{\prime}$. Note that $\Sigma^{\prime}$ extends to a representation $\Sigma$ of $W_\mathbb{F}$. After twisting, we may assume that $\Pi_x$ and $\Sigma|_{W_{\mathbb{F}_x}}$ are associated: Our results on Local Langlands show that there is a unique such choice of $\Sigma$.

Now $\Sigma$ gives rise to a system of $\ell$-adic representations $\sigma_{\ell}^1$, and $\Pi$ gives rise to a second system of (virtual) $\ell$-adic representations $\sigma_{\ell}^2$. They agree after restriction to $\mathrm{Gal}(\bar{\mathbb{F}}/\mathbb{F}^{\prime})$, giving the irreducible representation $\sigma_{\ell}^{\prime}$. Assume that $\ell$ is not divisible by $x$. Then there is a unique $\mathbb{Q}$-linear combination $\sigma_{\ell}$ of $\ell$-adic representations of $\mathrm{Gal}(\bar{\mathbb{F}}/\mathbb{F})$ with nonnegative coefficients such that the restriction to $\mathrm{Gal}(\bar{\mathbb{F}}/\mathbb{F}^{\prime})$ is $\sigma_{\ell}^{\prime}$, and such that $\sigma_{\ell}$ agrees locally at $x$ with $\sigma_{\ell}^1$: The virtual representation $\sigma_{\ell}$ is a sum of twists of $\sigma_{\ell}^1$ with characters of $\mathrm{Gal}(\mathbb{F}^{\prime}/\mathbb{F})$, and the coefficients of this sum can be read off by looking at the restriction to $x$, showing that only the coefficient of $\sigma_{\ell}^1$ is nonzero.

Hence we see that $\sigma_{\ell}^1 = \sigma_{\ell}^2$ for all $\ell$ which are not divisible by $x$. As above, this gives the result.
\end{proof}

Finally, we can prove the theorem on non-Galois automorphic induction.

\begin{thm}\label{AutomorphicInduction} Let $\mathbb{F}_0^3\supset \mathbb{F}_0^2\supset \mathbb{F}_0^1$ be extensions of totally real fields of even degree over $\mathbb{Q}$, such that $\mathbb{F}_0^3/\mathbb{F}_0^1$ is a solvable Galois extension. Write $n=[\mathbb{F}_0^2:\mathbb{F}_0^1]$. Further, let $\mathbb{K}$ be an imaginary-quadratic field. We let $\mathbb{F}^i=\mathbb{F}_0^i\mathbb{K}$ for $i=1,2,3$. Further, let $x$ be place of $\mathbb{F}^1$, which is split over $\mathbb{F}_0^1$ and inert in $\mathbb{F}^3$.

Assume that $\chi$ is a character of $\mathbb{F}^{2\times}\backslash \mathbb{A}_{\mathbb{F}^2}^{\times}$ such that
\begin{altenumerate}
\item[{\rm (i)}] $\chi^{-1} = \chi \circ c$,
\item[{\rm (ii)}] For any infinite place $\tau$ of $\mathbb{F}^2$, the component $\chi_{\tau}$ is of the form $z^{p_{\tau}}\overline{z}^{-p_{\tau}}$, where $p_{\tau}\in \mathbb{Z}$ and $p_{\tau}\neq p_{\tau^{\prime}}$ whenever $\tau\neq \tau^{\prime}$,
\item[{\rm (iii)}] The stabilizer of $\chi_x\circ \mathrm{Norm}_{\mathbb{F}^3/\mathbb{F}^2}$ in $\mathrm{Gal}(\mathbb{F}^3/\mathbb{F}^1)$ is equal to $\mathrm{Gal}(\mathbb{F}^3/\mathbb{F}^2)$.
\end{altenumerate}

Then there exists a potentially abelian cuspidal automorphic representation $I_{\mathbb{F}^2}^{\mathbb{F}^1} \chi$ of $\GL_n/\mathbb{F}^1$ associated to $\Sigma = \mathrm{Ind}_{W_{\mathbb{F}^2}}^{W_{\mathbb{F}^1}} \chi$.
\end{thm}

\begin{proof} First note that $\Pi=I_{\mathbb{F}^2}^{\mathbb{F}^1} \chi$ automatically has the properties $\Pi^{\vee}\cong \Pi\circ c$, $\Pi_{\infty}$ is regular $L$-algebraic, and $\Pi_x$ is supercuspidal: These follow from strong multiplicity $1$ and the compatibility at all places.

Now we proceed by induction on $[\mathbb{F}_0^3:\mathbb{F}_0^1]$. Choose a subextension $\mathbb{F}_0^3\supset \mathbb{F}_0^4\supset \mathbb{F}_0^1$ such that $\mathbb{F}_0^4/\mathbb{F}_0^1$ is cyclic Galois of prime degree. Write $\mathbb{F}^4 = \mathbb{F}_0^4\mathbb{K}$.

Assume first that $\mathbb{F}_0^4\subset \mathbb{F}_0^2$. Then we know that there exists $I_{\mathbb{F}^2}^{\mathbb{F}^4} \chi$. If we put $I_{\mathbb{F}^2}^{\mathbb{F}^1}\chi = \mathrm{Ind}_{\mathbb{F}^4}^{\mathbb{F}^1} I_{\mathbb{F}^2}^{\mathbb{F}^4}\chi$, where $\mathrm{Ind}$ denotes automorphic induction, \cite{ArthurClozel}, then we know that $\Pi=I_{\mathbb{F}^2}^{\mathbb{F}^1}\chi$ and $\mathrm{Ind}_{W_{\mathbb{F}^2}}^{W_{\mathbb{F}^1}} \chi$ are associated at all infinite places and almost all finite places, by the usual statements about automorphic induction. Looking at the place $x$, we see that $\Pi_x$ is supercuspidal, and hence $\Pi$ is cuspidal automorphic. Using Theorem \ref{InductionPreparation}, part (i), we are done in this case.

Now we are left with the case $\mathbb{F}_0^4\cap \mathbb{F}_0^2=\mathbb{F}_0^1$. By induction, we have a cuspidal automorphic representation $\Pi^{\prime} = I_{\mathbb{F}^2\mathbb{F}^4}^{\mathbb{F}^4} (\chi\circ \mathrm{Norm}_{\mathbb{F}^2\mathbb{F}^4/\mathbb{F}^2})$, associated to
\[
\Sigma^{\prime} = \mathrm{Ind}_{W_{\mathbb{F}^2\mathbb{F}^4}}^{W_{\mathbb{F}^4}} (\chi\circ \mathrm{Norm}_{\mathbb{F}^2\mathbb{F}^4/\mathbb{F}^2})= \Sigma|_{W_{\mathbb{F}^4}}\ .
\]
Note that our assumptions imply that $\Sigma^{\prime}_x$ is irreducible. By strong multiplicity $1$, $\Pi^{\prime}$ is invariant under $\mathrm{Gal}(\mathbb{F}^4/\mathbb{F}^1)$, hence descends to a cuspidal automorphic representation $\Pi$ of $\GL_n/\mathbb{F}^1$, unique up to twist. There is a unique choice of $\Pi$ such that $\Pi_x$ is associated to $\Sigma_x$, by our results on Local Langlands. Note that $\Pi_{\infty}$ is regular $L$-algebraic and $\Pi_x$ is supercuspidal, as $\Sigma_x$ is irreducible. Also $\Pi^{\vee}\cong (\Pi\circ c)\otimes \chi$ for some character $\chi$ of $\mathrm{Gal}(\mathbb{F}^4/\mathbb{F}^1)$. As in \cite{HarrisTaylor}, pp.241-242, one checks that there is some Artin character $\psi$ such that $\tilde{\Pi}=\Pi\otimes \psi$ satisfies $\tilde{\Pi}^{\vee}\cong \tilde{\Pi}\circ c$. Then Theorem \ref{InductionPreparation}, part (ii), shows that $\tilde{\Pi}$ is potentially abelian. Therefore $\Pi$ is potentially abelian, associated to a Weil group representation which is a twist of $\Sigma$ by a character of $\mathrm{Gal}(\mathbb{F}^4/\mathbb{F}^1)$. Looking at the place $x$, we see that it is equal to $\Sigma$.
\end{proof}

\section{Compatibility of $L$- and $\epsilon$-factors}

We return to the local setting. Let $\mathcal{A}_F$ be the free abelian group with generators given by the supercuspidal representations of $\GL_n(F)$, any $n\geq 1$, and let $\mathcal{G}_F$ be the Grothendieck group (with $\mathbb{Z}$-coefficients) of $W_F$. Recall that one can define $L$- and $\epsilon$-factors for elements of these groups, by extending linearly from the supercuspidal, resp. irreducible, representations.\footnote{A word of warning may be in order: One can send an irreducible smooth representation $\pi$ to its supercuspidal support, considered as an element of $\mathcal{A}_F$, and hence define $L(\pi,s)$, but this does not agree with the usual definition in general.} In the same way, one defines the central character $\omega_{\pi}$ for $\pi\in\mathcal{A}_F$, the determinant $\det \sigma$ for $\sigma\in \mathcal{G}_F$, the dual $\pi^{\vee}$ resp. $\sigma^{\vee}$, and the twist $\pi\otimes \chi$ resp. $\sigma\otimes\chi$ for a character $\chi$ of $F^{\times}$, using local class-field theory in the latter case.

Fix a nontrivial additive character $\psi: F\longrightarrow \mathbb{C}^{\times}$.

\begin{thm} The correspondence $\mathcal{A}_F\longrightarrow \mathcal{G}_F$ given by $\pi\longmapsto \sigma(\pi)$ has the following properties.
\begin{altenumerate}
\item[{\rm (i)}] In degree $1$, it is given by local class-field theory.
\item[{\rm (ii)}] It is compatible with twisting, i.e. $\sigma(\pi\otimes \chi)=\sigma(\pi)\otimes \chi$.
\item[{\rm (iii)}] It is compatible with duals, i.e. $\sigma(\pi^{\vee}) = \sigma(\pi)^{\vee}$.
\item[{\rm (iv)}] It is compatible with central characters, i.e. $\det \sigma(\pi)=\omega_{\pi}$.
\item[{\rm (v)}] It is compatible with $L$- and $\epsilon$-factors of pairs, i.e.
\[\begin{aligned}
L(\pi_1\times \pi_2,s) &= L(\sigma(\pi_1)\otimes \sigma(\pi_2),s)\\
\epsilon(\pi_1\times\pi_2,s,\psi) &= \epsilon(\sigma(\pi_1)\otimes \sigma(\pi_2),s,\psi)\ .
\end{aligned}\]
\end{altenumerate}
\end{thm}

\begin{proof} We have already seen parts (i) and (ii). By Brauer induction and linearity, it suffices to check all properties only for representations $\pi$ with $\sigma(\pi) = \mathrm{Ind}_{W_{F^{\prime}}}^{W_F} \chi$, where $\chi$ is a character of $W_{F^{\prime}}$. Fix two such representations $\pi_1$ and $\pi_2$.

\begin{lem} There exists a global field $\mathbb{F}$ with a place $w$ such that $F\cong \mathbb{F}_w$ and two potentially abelian cuspidal automorphic representations $\Pi_1$, $\Pi_2$ of $\GL_{n_i}/\mathbb{F}$ such that $(\Pi_i)_w$ is an unramified twist of $\pi_i$ for $i=1,2$, and for which the associated Weil group representations $\Sigma_i$ are irreducible.
\end{lem}

\begin{proof} This follows from Theorem \ref{AutomorphicInduction} after embedding everything into the global picture (Note that the Weil group representations that occur are irreducible, as they are irreducible at $x$.). That this can be done is checked e.g. in \cite{HarrisTaylor}, proof of Lemma VII.2.10, and \cite{HenniartLLC}, Section 3.
\end{proof}

Note that this implies statement (iv), since $\det \Sigma_1 = \omega_{\Pi_1}$, e.g. by strong multiplicity $1$. To prove statement (v), one uses the functional equation for $L$-functions of pairs, together with the trick of twisting with highly ramified characters, cf. Corollary 2.4 of \cite{HenniartLLC}. Finally, part (iii) follows from part (v) by looking at poles of $L$-functions.
\end{proof}

\bibliographystyle{abbrv}
\bibliography{LocalLanglands}

\def\cprime{$'$}
\begin{thebibliography}{10}

\bibitem{ArthurClozel}
J.~Arthur and L.~Clozel.
\newblock {\em Simple algebras, base change, and the advanced theory of the
  trace formula}, volume 120 of {\em Annals of Mathematics Studies}.
\newblock Princeton University Press, Princeton, NJ, 1989.

\bibitem{ArtinAlgebraization}
M.~Artin.
\newblock Algebraization of formal moduli. {I}.
\newblock In {\em Global {A}nalysis ({P}apers in {H}onor of {K}. {K}odaira)},
  pages 21--71. Univ. Tokyo Press, Tokyo, 1969.

\bibitem{Berkovich2}
V.~G. Berkovich.
\newblock Vanishing cycles for formal schemes. {II}.
\newblock {\em Invent. Math.}, 125(2):367--390, 1996.

\bibitem{BernsteinZelevinsky}
I.~N. Bernstein and A.~V. Zelevinsky.
\newblock Induced representations of reductive {${\mathfrak p}$}-adic groups.
  {I}.
\newblock {\em Ann. Sci. \'Ecole Norm. Sup. (4)}, 10(4):441--472, 1977.

\bibitem{BuzzardGee}
K.~Buzzard and T.~Gee.
\newblock The conjectural connections between automorphic representations and
  {G}alois representations.
\newblock arXiv:1009.0785.

\bibitem{Drinfeld}
V.~G. Drinfel{\cprime}d.
\newblock Elliptic modules.
\newblock {\em Mat. Sb. (N.S.)}, 94(136):594--627, 656, 1974.

\bibitem{Faltings}
G.~Faltings.
\newblock Group schemes with strict {$\mathcal O$}-action.
\newblock {\em Mosc. Math. J.}, 2(2):249--279, 2002.
\newblock Dedicated to Yuri I. Manin on the occasion of his 65th birthday.

\bibitem{FarguesAbelianMotives}
L.~Fargues.
\newblock Motives and automorphic forms: the abelian case.
\newblock http://www.math.u-psud.fr/ {$\tilde{}$} fargues/Motifs\_abeliens.ps.

\bibitem{HarrisNonGalois}
M.~Harris.
\newblock The local {L}anglands conjecture for {${\rm GL}(n)$} over a
  {$p$}-adic field, {$n<p$}.
\newblock {\em Invent. Math.}, 134(1):177--210, 1998.

\bibitem{HarrisTaylor}
M.~Harris and R.~Taylor.
\newblock {\em The geometry and cohomology of some simple {S}himura varieties},
  volume 151 of {\em Annals of Mathematics Studies}.
\newblock Princeton University Press, Princeton, NJ, 2001.
\newblock With an appendix by Vladimir G. Berkovich.

\bibitem{HenniartNumericalLLC}
G.~Henniart.
\newblock La conjecture de {L}anglands locale num\'erique pour {${\rm GL}(n)$}.
\newblock {\em Ann. Sci. \'Ecole Norm. Sup. (4)}, 21(4):497--544, 1988.

\bibitem{HenniartUniqueness}
G.~Henniart.
\newblock Caract\'erisation de la correspondance de {L}anglands locale par les
  facteurs {$\epsilon$} de paires.
\newblock {\em Invent. Math.}, 113(2):339--350, 1993.

\bibitem{HenniartLLC}
G.~Henniart.
\newblock Une preuve simple des conjectures de {L}anglands pour {${\rm GL}(n)$}
  sur un corps {$p$}-adique.
\newblock {\em Invent. Math.}, 139(2):439--455, 2000.

\bibitem{JacquetPiatetskiShapiroShalika}
H.~Jacquet, I.~I. Piatetski-Shapiro, and J.~Shalika.
\newblock Conducteur des repr\'esentations du groupe lin\'eaire.
\newblock {\em Math. Ann.}, 256(2):199--214, 1981.

\bibitem{Kazhdan}
D.~Kazhdan.
\newblock Cuspidal geometry of {$p$}-adic groups.
\newblock {\em J. Analyse Math.}, 47:1--36, 1986.

\bibitem{KottwitzLambdaAdic}
R.~E. Kottwitz.
\newblock On the {$\lambda$}-adic representations associated to some simple
  {S}himura varieties.
\newblock {\em Invent. Math.}, 108(3):653--665, 1992.

\bibitem{KottwitzPoints}
R.~E. Kottwitz.
\newblock Points on some {S}himura varieties over finite fields.
\newblock {\em J. Amer. Math. Soc.}, 5(2):373--444, 1992.

\bibitem{LanglandsTori}
R.~P. Langlands.
\newblock Representations of abelian algebraic groups.
\newblock {\em Pacific J. Math.}, (Special Issue):231--250, 1997.
\newblock Olga Taussky-Todd: in memoriam.

\bibitem{Mantovan1}
E.~Mantovan.
\newblock On certain unitary group {S}himura varieties.
\newblock {\em Ast\'erisque}, (291):201--331, 2004.
\newblock Vari{\'e}t{\'e}s de Shimura, espaces de Rapoport-Zink et
  correspondances de Langlands locales.

\bibitem{Mantovan2}
E.~Mantovan.
\newblock On the cohomology of certain {PEL}-type {S}himura varieties.
\newblock {\em Duke Math. J.}, 129(3):573--610, 2005.

\bibitem{Mieda}
Y.~Mieda.
\newblock On {$l$}-independence for the \'etale cohomology of rigid spaces over
  local fields.
\newblock {\em Compos. Math.}, 143(2):393--422, 2007.

\bibitem{SchneiderZinkTypes}
P.~Schneider and E.-W. Zink.
\newblock {$K$}-types for the tempered components of a {$p$}-adic general
  linear group.
\newblock {\em J. Reine Angew. Math.}, 517:161--208, 1999.
\newblock With an appendix by Schneider and U. Stuhler.

\bibitem{ScholzeDeformationSpaces}
P.~Scholze.
\newblock The {L}anglands-{K}ottwitz approach and deformation spaces of
  {$p$}-divisible groups.
\newblock 2010.
\newblock in preparation.

\bibitem{ScholzeGLn}
P.~Scholze.
\newblock The {L}anglands-{K}ottwitz approach for some simple {S}himura
  varieties.
\newblock 2010.
\newblock \url{arXiv:1003.2451}.

\bibitem{Scholze}
P.~Scholze.
\newblock The {L}anglands-{K}ottwitz approach for the modular curve.
\newblock 2010.
\newblock \url{arXiv:1003.1935}.

\bibitem{Shin}
S.~W. Shin.
\newblock Counting points on {I}gusa varieties.
\newblock {\em Duke Math. J.}, 146(3):509--568, 2009.

\bibitem{Strauch}
M.~Strauch.
\newblock On the {J}acquet-{L}anglands correspondence in the cohomology of the
  {L}ubin-{T}ate deformation tower.
\newblock {\em Ast\'erisque}, (298):391--410, 2005.
\newblock Automorphic forms. I.

\bibitem{Varshavsky}
Y.~Varshavsky.
\newblock Lefschetz-{V}erdier trace formula and a generalization of a theorem
  of {F}ujiwara.
\newblock {\em Geom. Funct. Anal.}, 17(1):271--319, 2007.

\bibitem{Zelevinsky}
A.~V. Zelevinsky.
\newblock Induced representations of reductive {${\mathfrak p}$}-adic groups.
  {II}. {O}n irreducible representations of {${\rm GL}(n)$}.
\newblock {\em Ann. Sci. \'Ecole Norm. Sup. (4)}, 13(2):165--210, 1980.

\end{thebibliography}

\end{document}